\documentclass[16pt]{amsart}

\usepackage{latexsym,enumerate}
\usepackage{amsmath,amsthm,amsopn,amstext,amscd,amsfonts,amssymb}
\usepackage[ansinew]{inputenc}
\usepackage{verbatim}
\usepackage{graphicx}
\usepackage{pstricks}
\usepackage{wasysym}
\usepackage[all]{xy}

\usepackage{epsfig} % Para graficos
\usepackage{graphicx,psfrag} % Mejor este para    graficos: es mas moderno

\usepackage{multirow, array} % para las tablas
\usepackage{float} % para usar [H]
\usepackage{tikz}
\usepackage{caption}
\usepackage{subcaption}  %para poner figuras en columnas o lineas diferentes

\usepackage{pstricks,pst-node}
\usepackage{multicol}
\usepackage{color}
\usepackage{colortbl}

\graphicspath{ {Imgs/} } %toma las imagenes d ela carpeta Imgs

\newcommand{\NN}{\mathbb{N}}

\newtheorem{theorem}{Theorem}[section]
\newtheorem{lemma}[theorem]{Lemma}
\newtheorem{proposition}[theorem]{Proposition}
\newtheorem{corollary}[theorem]{Corollary}
\newtheorem{definition}[theorem]{Definition}

\newtheorem{remark}[theorem]{Remark}

\newcommand{\spb}[1]{\smallskip}
\newcommand{\mpb}[1]{\medskip}
\newcommand{\bpb}[1]{\bigskip}

\newcommand{\e}{\varepsilon}
\renewcommand{\d}{\delta}

\newcommand{\dx}{\delta ^*(x,y,z,w)}
\newcommand{\pr}{^{\prime}}

\begin{document}
\DeclareGraphicsExtensions{.jpg,.pdf,.mps,.png}
 
\Large %%%%%%%%%%%%%%%%%%%%%%%%%%%%%%%%%%%%%%%%%%%%%%%%%%%%%%%%%%%%%%%%%%% aumenta tama0Š9o de letra

\title{The Quarter-Point Theorem for Gromov Hyperbolic Graphs}
%\title{Gromov product on graphs}
%\title{Hyperbolic chemical graphs}

\author[W. Carballosa]{Walter Carballosa$^{(a)}$}
\address{$^{(a)}$Florida International University}
%\email{wcarball@fiu.edu}

\author[J. Mendez]{Jesus Mendez$^{(b)}$}
\address{$^{(b)}$Facultad de Matem\'aticas, Universidad Aut\'onoma de Guerrero,
Carlos E. Adame No.54 Col. Garita, 39650 Acapulco Gro., Mexico}
%\email{jangelmendezb.1996@gmail.com}

\author[R. Reyes]{Rosalio Reyes$^{(c), (d)}$}
\address{$^{(c)}$Instituto de F\'{\i}sica, Benem\'erita Universidad Aut\'onoma de Puebla,
Apartado Postal J-48, Puebla 72570, Mexico}
%\email{rreyes@ifuap.buap.mx}
\address{$^{(d)}$Universidad Carlos III de Madrid, ROR: https://ror.org/03ths8210, Departamento de Matem\'aticas, Avenida de la Universidad, 30 (edificio Sabatini), 28911, Legan\'es (Madrid), Spain, ORCID: 0000-0002-1702-0003}			 

%\author[Jos\'e M. Rodr{\'\i}guez]{Jos\'e M. Rodr{\'\i}guez$^{(1)}$}
%\address{Departamento de Matem\'aticas, Universidad Carlos III de Madrid,
%Avenida de la Universidad 30, 28911 Legan\'es, Madrid, Spain}
%\email{jomaro@math.uc3m.es}
%\thanks{$^{(1)}$ Supported by a grant from Agencia Estatal de Investigaci\'on (PID2019-106433GB-I00 / AEI / 10.13039/501100011033), Spain.}

\author[Jos\'e M. Sigarreta]{Jos\'e M. Sigarreta$^{(b),*}$}
%\address{Facultad de Matem\'aticas, Universidad Aut\'onoma de Guerrero, Carlos E. Adame No.54 Col. Garita, 39650 Acapulco Gro., Mexico}
\address{$^*$Corresponding author}
\email{josemariasigarretaalmira@hotmail.com}

%\date{\today}

\maketitle{}

%\centerline{{\bf Sierpinski 02.tex}}

\begin{abstract}
%The concept of Gromov hyperbolicity grasps the essence of negatively curved spaces like the classical hyperbolic space and Riemannian manifolds of negative sectional curvature. It is remarkable that a simple concept leads to such a rich general theory. The study of hyperbolic graphs is an interesting topic since the hyperbolicity of any geodesic metric space is equivalent to the hyperbolicity of a graph related to it.
Gromov hyperbolicity measures the geometric complexity of metric spaces and metric graphs through geodesic triangles and four-point conditions. While calculating the exact continuous hyperbolicity constant $\delta^*(G)$ of a graph $G$ typically requires evaluating quadruples of points over an infinite metric continuum, discretizing this domain is a fundamental challenge in metric graph theory. In this paper, we study the continuous Gromov hyperbolicity constant via the Gromov product. Our main result, the Quarter-Point Theorem, establishes that the search space for quadruples $(x,y,z,w)$ attaining the exact hyperbolicity constant $\delta^*(G)$ can be reduced from the entire metric graph to a finite or discrete subset $J'(G)$, comprising solely the vertices, edge midpoints, and quarter-points of $G$. This result holds for every graph, finite or infinite. This theoretical result provides a considerably simplified combinatorial framework for calculating hyperbolicity in complex networks and metric graphs. Furthermore, we exploit this discretization to significantly improve classical results by establishing and structural bounds relating the global hyperbolicity constant $\delta^*(G)$ with its restrictions to vertices ($\delta_V^*(G)$) and midpoints ($\delta_J^*(G)$). % These theoretical tools provide a significantly simplified combinatorial framework for computing hyperbolicity in complex networks and metric graphs.
\end{abstract}


%{\it 2010 AMS Subject Classification Numbers:} %05C07, 92E10, 05C90, 05C12, 05C75.

\section{Introduction}

Hyperbolicity in Gromov's sense is a simple concept, but at the same time it leads to a very rich general theory. 
This concept captures the qualities of negatively curved spaces like the classical hyperbolic space, Riemannian manifolds of negative sectional curvature, and  discrete spaces like trees and Cayley graphs of many finitely generated groups.
Hyperbolicity is a very useful tool that helps us in understanding the relationships between graphs and  manifolds.  
It is known from the works of Gromov and Kanai that  graphs can model  manifolds and many metric spaces (this is an interesting result, since it allows to go from working with a continuous structure that could be complicated to deal with a discrete structure, see, e.g., \cite{Ghys,G1}). 
In \cite{PRT1,T} the equivalence between the hyperbolicity of many surfaces and the  hyperbolicity of simple graphs is proved.

One of the first applications of the Gromov hyperbolic spaces was the analysis of finitely generated groups (see, \emph{e.g.}, \cite{O}). Today, %the mathematical properties and various applications of these spaces are being studied,
the mathematical properties and the applications of Gromov hyperbolicity are subjects of study with a growing interest within graph theory (see  \cite{BRS,BRSV,BeRoSiVi1,BeRoSiTo,BHB1,CPRS,CRSV,K50,GL,K27,K22,K24,KoMo,MiRoSiVi1,MiRoSiVi2,PRSV2,PRT1,PRT2,PT,RS,RSVV,T,WZ} and the references therein).

There are several definitions of Gromov hyperbolicity which are equivalent, in the sense that if $X$ is $\d$-hyperbolic with respect to the definition $A$, then it is $\d'$-hyperbolic with respect to the definition $B$ for some $\d'$ (see, e.g., \cite{BHB, ABCD, Ghys}).

Let $(X,d)$ be a metric space and  $ x,y \in X $. The \emph{Gromov product} between $ x$ and $y$, with base point $w \in X$, is defined as

%$$ (x \mid y )_w :=  {1 \over 2} \big( d(x,w)+d(y,w)- d(x,y) \big ) \geq 0. $$
$$ (x , y )_w :=  {1 \over 2} \big( d(x,w)+d(y,w)- d(x,y) \big ) \geq 0. $$

\noindent We say that the metric space $(X,d)$ is $ \delta$-\emph{hyperbolic} with respect to the  Gromov product, for some constant $ \delta \geq 0  $, if

%$$ (x \mid z )_w \geq \min \{ (x \mid y )_w, (y \mid z )_w \} - \delta $$
%$$ (x , z )_w \geq \min \{ (x , y )_w, (y , z )_w \} - \delta $$
\begin{equation*}
\label{eq:product}
(x,z)_w
\ge \min \big\{ (x,y)_w, (y,z)_w \big\}- \d
\end{equation*}

\noindent for every $x,y,z,w \in X $.

%We say that $X$ is \emph{hyperbolic} if $X$ is $\delta$-hyperbolic for some $\delta \geq 0$; then $X$ is hyperbolic if and only if $ \delta(X)<\infty.$ If $X$ has connected components $\{X_i\}_{i\in I}$, then we define $\delta(X):=\sup_{i\in I} \d(X_i)$, and we say that $X$ is hyperbolic if $\d(X)<\infty$.

%%%%%%%%%%%%%%%%%%%%%%%%%%%%%%%%%55
%%%%%%%%%%%%%%%%%%%%%%%%%%
\begin{definition}\label{def:finite-hyp}
Let $G$ be a graph. We say that $G$ is \emph{hyperbolic} (with respect to the Gromov product) if
$$
\delta^*(G) := \sup\{\delta^*(x,y,z,w) : x,y,z,w \in G\} < \infty.
$$
\end{definition}

%\begin{proposition}\label{prop:equiv-finite}
%Let $G$ be a finite graph. The following statements are equivalent:
%\begin{enumerate}
%\item[(i)] $G$ is hyperbolic in the sense of Definition~\ref{def:finite-hyp}, i.e.\ $\delta^*(G) < \infty$;
%\item[(ii)] there exist $x,y,z,w \in G$ such that $\delta^*(G) = \delta^*(x,y,z,w)$.
%\end{enumerate}
%Moreover, both (i) and (ii) hold automatically for \emph{every} finite graph $G$.
%\end{proposition}

%\begin{proof}
%This is precisely the content of Lemma~\ref{lem:attainment}: since $G$ is finite, $\delta^*(\cdot,\cdot,\cdot,\cdot)$ is a continuous function on the compact space $G^4$, so its supremum is both finite (giving (i)) and attained (giving (ii)) simultaneously. In particular (ii) $\Rightarrow$ (i) is immediate, and (i) $\Rightarrow$ (ii) follows from Lemma~\ref{lem:attainment}, which does not even require (i) as a hypothesis.
%\end{proof}

%%%%%%%%%%%%%%%%%%%%%%%%%%%%%%%%%%%%%%5
%%%%%%%%%%%%%%%%%%%%%%%%%%%%%%%%

\medskip

Throughout this paper, $G=(V,E)=\left( V\left( G \right),E\left(G \right) \right)$ denotes a (finite or infinite) simple (without loops and multiple edges) graph (not necessarily connected) such that $V\neq \emptyset$ and every edge has length $1$. In order to consider a graph $G$ as a geodesic metric space, identify (by an isometry) any edge $uv\in E(G)$ with the interval $[0,1]$ in the real line; then the edge $uv$ (considered as a graph with just one edge) is isometric to the interval $[0,1]$. Thus, the points in $G$ are the vertices and, also, the points in the interior of any edge of $G$. In this way, any connected graph $G$ has a natural distance defined on its points, induced by taking shortest paths in $G$, and we can see $G$ as a metric graph. We denote by $d_G$ or $d$ this distance. If $x,y$ are in different connected components of $G$, we define $d_G(x,y)=\infty$. These properties guarantee that any connected component of any graph is a geodesic metric space.

Let $G$ be a graph and  $x,\, y, \, z, \, w \, \in G$, we denote by $\dx$ the hyperbolicity constant associated to $x,y,z,w$, \emph{i.e.}
$$\dx = \max_{v_i \in \{x, y,z,w \} } \left\lbrace  \min \{ (v_1,v_2)_{v_4},(v_2,v_3)_{v_4} \} -(v_1,v_3)_{v_4} \right\rbrace. $$ %A menos que se especifique lo contrario Consideraremos  $\dx \geq 0$.

Here the maximum ranges over all ways of choosing a base point $v_4 \in \{x,y,z,w\}$ together with an assignment of the remaining three points to the roles $v_1,v_2,v_3$; since the expression $\min\{(v_1,v_2)_{v_4},(v_2,v_3)_{v_4}\}-(v_1,v_3)_{v_4}$ is invariant under swapping $v_1$ and $v_3$, this yields, for each choice of base point, three distinct values to compare, so that $\dx$ is the maximum over the resulting (at most twelve) quantities. Throughout the paper, whenever we write $\d^*(x,y,z,w)=(x,y)_w-(x,z)_w$ we mean that this particular term, with base point $w$, attains the maximum defining $\dx$.

Given a graph $G$, we denote by $J(G)$ the union of the set $V(G)$ and the midpoints of the edges
of $G$. We denote by $J'(G)$ the union of $J(G)$ and  quarter points of the edges
of $G$.

\medskip

%{ \color{red}

If $X$ is a geodesic metric space and $x_1,x_2,x_3\in X$, the union
of three geodesics $[x_1 x_2]$, $[x_2 x_3]$ and $[x_3 x_1]$ is a
\emph{geodesic triangle} that will be denoted by $T=\{x_1,x_2,x_3\}$
and we will say that $x_1,x_2$ and $x_3$ are the vertices of $T$; it
is usual to write also $T=\{[x_1x_2], [x_2x_3], [x_3x_1]\}$. We say
that $T$ is $\d$-{\it thin} if any side of $T$ is contained in the
$\d$-neighborhood of the union of the two other sides. We denote by
$\d(T)$ the sharp thin constant of $T$, i.e., $ \d(T):=\inf\{\d\ge 0\,|\,
\, T \, \text{ is $\d$-thin}\}. $ The space $X$ is
$\d$-\emph{hyperbolic} $($or satisfies the {\it Rips condition} with
constant $\d)$ if every geodesic triangle in $X$ is $\d$-thin.
If we have a triangle with two
identical vertices, we call it a ``bigon"$\!$. Obviously, every bigon in
a $\d$-hyperbolic space is $\d$-thin.
We denote by $\d(X)$ the sharp hyperbolicity constant of $X$, i.e.,
$\d(X):=\sup\{\d(T)\,|\, \, T \, \text{ is a geodesic triangle in }X \}.$
We say that $X$ is \emph{hyperbolic} if $X$ is $\d$-hyperbolic for some $\d \ge 0$; then $X$ is hyperbolic if and only if $ \d(X)<\infty.$
If $X$ has connected components $\{X_i\}_{i\in I}$, then we define $\d(X):=\sup_{i\in I} \d(X_i)$, and we say that $X$ is hyperbolic if $\d(X)<\infty$.
%} %Termina color red

 In \cite{BRS}, the authors show that when studying hyperbolicity, with Rips' condition, in graphs, it is sufficient to consider only the set of vertices and midpoints of edges of the graph. In this work, we consider hyperbolicity with respect to Gromov's product and show that for this definition, it is sufficient to study the vertices of the graph, as well as the midpoints and quarter points of the edges.
 When working with graphs as metric spaces, when considering hyperbolicity with Rips' condition, we work with triangles whose vertices are graph vertices or midpoints of edges, and in which we must specify the paths followed by geodesics. whereas with Gromov's product we work with quadrilaterals on a larger set, but it is not necessary to specify the paths followed by the geodesics. This presents an advantage when considering transformations or operations on the graph, which can significantly distort the paths but the distances do not vary much from the original ones.

%Dada una gr¨¢fica $G$ denotamos por $J(G)$ el conjunto de v¨¦rtices  y puntos medios de aristas. Adem¨¢s denotamos por $J'(G)$ el conjunto de v¨¦rtices, puntos medios y puntos cuartos de aristas. Si $S \subset G$ denotamos por $J'(S)$ el conjunto de v¨¦rtices, puntos medios y puntos cuartos de aristas contenidos en $S$.

%\subsection{Definiciones} \label{Sec:def}

\section{Quarter-point theorem}

%%%%%%%%%%%%%%%%%%%%%%%%%%%%%%%%%%%%%%%%%%%%%%%%%%%%%%%%%%%%
%%%%%%%%%%%%%%%%%%%%%%%%%%%%%%%%%%%%%%%%%%%%%%%%%%%%%%%%%%%%%%%%%%%%%%%%%%%%%%%%%
%%%%%%%%%%%%%%%%%%%%%%%%%%%%%%%%%%%%%%%%%%%%%%%%%%%%%%%%%%%%%%%%%%%%%%%%%%%%%%%%%%%%%%%%%%%%%%%%%%%

\begin{theorem}\label{th:finitos}
Let $G = (V, E)$ be a finite graph, where each edge $e \in E$ is identified with an isometric interval $[0, 1]$. Let $f: G^4 \to \mathbb{R}$ be defined by
\[
f(x, y, z, w) = \min\left\{ (x, y)_w, (y, z)_w \right\} - (x, z)_w.
\]
Then, the Gromov hyperbolicity constant $\delta^*(G) := \sup_{x, y, z, w \in G} f(x, y, z, w)$ is attained as a maximum; that is, there exists $(x_0, y_0, z_0, w_0) \in G^4$ such that $f(x_0, y_0, z_0, w_0) = \delta^*(G)$.
\end{theorem}

\begin{proof}
Since $G$ is finite, $G$ is a finite union of compact metric segments $[0,1]$ glued along a finite set of vertices $V$, making $G$ a compact metric space. By Tychonoff's Theorem, the product space $G^4$ equipped with the product metric is compact.

The geodesic distance $d: G \times G \to \mathbb{R}_{\ge 0}$ is continuous, which implies that the Gromov product $(a, b)_w = \frac{1}{2}(d(a,w) + d(b,w) - d(a,b))$ is continuous on $G^3$. Since $f$ is a linear combination and minimum of continuous Gromov products, $f: G^4 \to \mathbb{R}$ is continuous.

By the Extreme Value Theorem, a continuous real-valued function on a compact space attains its maximum. Hence, there exists $(x_0, y_0, z_0, w_0) \in G^4$ such that $f(x_0, y_0, z_0, w_0) = \delta^*(G)$.
\end{proof}

%%%%%%%%%%%%%%%%%%%%%%%%%%%%%%%%%%%%%%%%%%%%%%%%%%%%%%%%%%%%%%%%%%%%%%%%%%%%%%%%%%%%%%%%%
%%%%%%%%%%%%%%%%%%%%%%%%%%%%%%%%%%%%%%%%%%%%%%%%%%%%%%%%%%%%%%%%%%%%%%%%%%%%%%%%%%%%%%%%%
%%%%%%%%%%%%%%%%%%%%%%%%%%%%%%%%%%%%%%%%%%%%%%%%%%%%%%%%%%%%%%%%%%%%%%%%%%%%%%%%%%%%%%

The following result is quite useful, as it indicates a very important property about the distribution of geodesics (see Figure \ref{fig:base}).

%Figura 1
\begin{figure}[htb]
\centering
\includegraphics[width=0.5\textwidth]{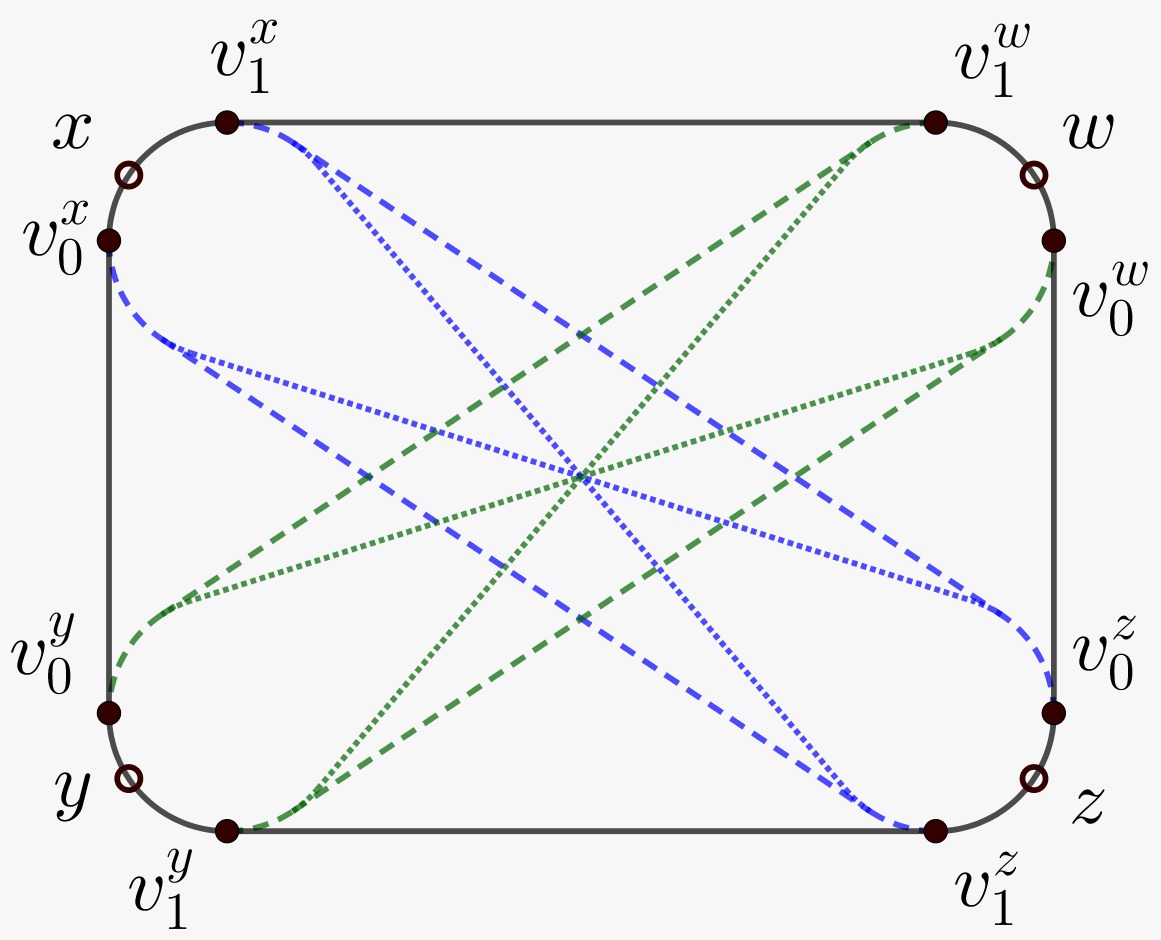}
\caption{If $\d^*(x,y,z,w)=(x,y)_w-(x,z)_w$ then we can consider the geodesics $[xy]$, $[yz]$, $[zw]$ and $[wx]$ as shown in the image, therefore, we will study only the cases for the geodesics $[xz]$(blue) and $[yw]$(green).}
\label{fig:base}
\end{figure}

\begin{proposition} \label{Prop:vinxyxw}
Let $G$ be a graph and $x,y,z,w \in G$ such that $\delta^*(x,y,z,w) = (x,y)_w - (x,z)_w$. If at least one of the following points exists:
 $x' \in [xy] \cap [xw]$,  $y' \in [xy] \cap [yz]$, $z' \in [zw] \cap [yz]$ or  $w' \in [zw] \cap [xw]$,
then the following inequality holds:
$$\delta^*(x^*,y^*,z^*,w^*) \geq \delta^*(x,y,z,w)$$
for any choice of $x^*\in \{x, x' \}$, $y^*\in \{y, y' \}$, $z^* \in \{z, z' \}$, and $w^* \in \{w, w' \}$.
\end{proposition}
\begin{proof}
Note that $(y,z)_w \geq (x,y)_w$. If $x' \in [xy] \cap [xw]$ (see Figure \ref{fig:Dist_x}),  we have 
$d(x,y)=d(x,x')+d(x',y)$ and $d(x,w)=d(x,x')+d(x',w)$  and by triangle inequality we obtain  $d(x,z) \leq d(x,x')+d(x',z)$. Then  
\begin{align*}
(x',y)_w-(x',z)_w&= \frac{1}{2} \left( d(x',w)+d(y,w)-d(x',y)-d(x',w)-d(z,w)+d(x',z)\right) \\
&=\frac{1}{2} \left( d(y,w)-d(x',y)-d(z,w)+d(x',z) \right) \\
&\geq \frac{1}{2} \left( d(y,w)-(d(x,y)-d(x,x'))-d(z,w)+d(x,z)-d(x,x')\right) \\
&= \frac{1}{2} \left( d(y,w)-d(x,y)-d(z,w)+d(x,z) \right) \\
&= (x,y)_w-(x,z)_w. \\
(y,z)_w-(x',z)_w&=\frac{1}{2} \left( d(y,w)+d(z,w)-d(y,z)-d(x',w)-d(z,w)+d(x',z) \right) \\
&=\frac{1}{2} \left(  d(y,w)-d(y,z)-d(x',w)+d(x',z)  \right) \\
&\geq \frac{1}{2} \left(  d(y,w)-d(y,z)-(d(x,w)-d(x,x'))+d(x,z) -d(x,x') \right) \\
&=\frac{1}{2} \left( d(y,w)-d(y,z)-d(x,w) +d(x,z) \right) \\
& =(y,z)_w-(x,z)_w \geq (x,y)_w-(x,z)_w.
\end{align*}
Hence, 
$
\d^*(x',y,z,w) \geq \min \{ (x',y)_w, \, (y,z)_w \} -(x',z)_w \geq (x,y)_w-(x,z)_w =\d^*(x,y,z,w).
$

If $y\pr \in [xy] \cap [yz]$, following a similar procedure and considering $d(x,y)=d(x,y')+d(y',y)$, $d(y,z)=d(y,y')+d(y',z)$ and $d(y,w) \leq d(y,y')+ d(y',w)$, we obtain 
$
\d^*(x,y',z,w) \geq \min \{ (x,y')_w, \, (y',z)_w \} -(x,z)_w \geq (x,y)_w-(x,z)_w =\d^*(x,y,z,w).
$

If $z\pr \in [zw] \cap [yz]$ consider $d(y,z)=d(y,z')+d(z',z)$, $d(w,z)=d(w,z')+d(z',z)$ and $d(x,z) \leq d(x,z')+ d(z',z)$. Following a similar procedure, we obtain $
\d^*(x,y,z',w) \geq \min \{ (x,y)_w, \, (y,z')_w \} -(x,z')_w \geq (x,y)_w-(x,z)_w =\d^*(x,y,z,w).
$

Finally, if $w\pr \in [zw] \cap [xw]$ then considering $d(w,z)=d(w,w')+d(w',z)$, $d(x,w)=d(x,w')+d(w',w)$ and $d(y,w) \leq d(y,w')+ d(w',w)$ and following a similar way, we obtain $\d^*(x,y,z,w') \geq \min \{ (x,y)_{w'}, \, (y,z)_{w'} \} -(x,z)_{w'} \geq (x,y)_w-(x,z)_w =\d^*(x,y,z,w)$.
\end{proof}

\begin{figure}[htbp]
\centering
\begin{tikzpicture}[
    scale=1.5, 
    important node/.style={circle, fill=black, inner sep=1.8pt},
    variable node/.style={circle, fill=red, inner sep=2pt} % Destaca x' si lo deseas
]

    % Definici¨®n de las coordenadas de los puntos
    \coordinate (x)  at (0.5, 2.2);
    \coordinate (xp) at (1.5, 1.8);
    \coordinate (w)  at (4.0, 2.5);
    \coordinate (y)  at (0.0, 0.0);
    \coordinate (z)  at (5.0, 0.0);

    % Dibujo de las conexiones (Aristas)
    \draw[thick] (x) -- (xp);   % Conexi¨®n x a x'
    \draw[thick] (xp) -- (y);  % Conexi¨®n x' a y
    \draw[thick] (xp) -- (w);  % Conexi¨®n x' a w
    \draw[thick] (y) -- (z);   % Conexi¨®n y a z
    \draw[thick] (z) -- (w);   % Conexi¨®n z a w

    % Dibujo de los puntos (Nodos)
    \node[important node] at (x) {};
    \node[important node] at (xp) {};
    \node[important node] at (y) {};
    \node[important node] at (z) {};
    \node[important node] at (w) {};

    % Etiquetas de texto bien posicionadas para que no se superpongan
    \node[above left]  at (x)  {$x$};
    \node[above right] at (xp) {$x'$};
    \node[above right] at (y)  {$y$};
    \node[below left]  at (z)  {$z$};
    \node[below right] at (w)  {$w$};

\end{tikzpicture}
\caption{If $x' \in [xy] \cap [xw]$ then $\delta^*(x',y,z,w) \geq \delta (x,y,z,w)$.}
\label{fig:Dist_x}
\end{figure}
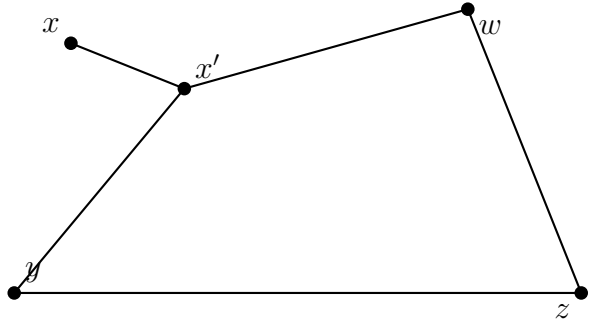

Let us consider $x,y,z,w \in G \setminus V(G)$ with  $\delta^*(x,y,z,w)=(x,y)_w-(x,z)_w$.  Proposition \ref{Prop:vinxyxw} gives an idea of how to consider the geodesics  $[xy], [yz], [zw]$ and $[wx]$.   If $v^x_0v^x_1, \, v^y_0v^y_1,  \,  v^z_0v^z_1, \, v^w_0v^w_1 \in E(G)$  contain $x, y, z, w$ respectively, with $v^x_0, v^y_0 \in [xy]$ and  $v^z_0, v^w_0 \in [zw]$, then, from now on, we will consider  $v^x_1, v^w_1 \in [xw]$ and $v^y_1, v^w_1 \in [yw]$ (see Figure \ref{fig:base}).

Since we will consider elements of $J'(G)$ the following result will be  useful  (see Figure \ref{fig:23}).

\medskip

\begin{proposition}\label{Th:Pcuarto-Pgen}
Let $G$ be a graph, $v',u'_1,u'_2 \in J'(G)$ such that $d(u'_1,u'_2)=1/4$ and $x \in [u'_1u'_2] \setminus \{u'_1,u'_2 \}$. Consider the geodesic $[v'x]$, if $u'_1 \in [v'x]$ then $d(v',u'_2)> d(v',x)$, in particular $ d(v',u'_2)= d(v',x) +d(x,u'_2)$.
\end{proposition}
\begin{proof}
Consider $d(v',u'_1)=n+k/4$ with $n$ integer and $k \in \{0,1,2,3 \}$.
Suppose that $d(v',u'_2) \leq d(v',x)$.
If $d(v',u'_2) = d(v',x)$ consider the cycle $C=[v'u'_1] \cup [u'_1u'_2] \cup [u'_2v']$, the length  of $C$ is integer therefore $(2k+1)/4$ is integer but it is not possible. 
If $d(v',u'_2) < d(v',x)$ then  $d(v',u'_2) \leq n+(k-1)/4$. Moreover, note that  $\max \{d(x,u'_1),d(x,u'_2) \} <1/4$ therefore 
$$
d(v',u'_2)+d(u'_2,x)< n+(k-1)/4 + 1/4 = n+k/4 < d(v',u'_1)+d(u'_1,x)
$$
which contradicts $[v',x]=[v',u'_1]\cup[u'_1,x]$ is a geodesic. 
\end{proof}

%Figura 2
\begin{figure}[htb]
\centering
\includegraphics[width=0.5\textwidth]{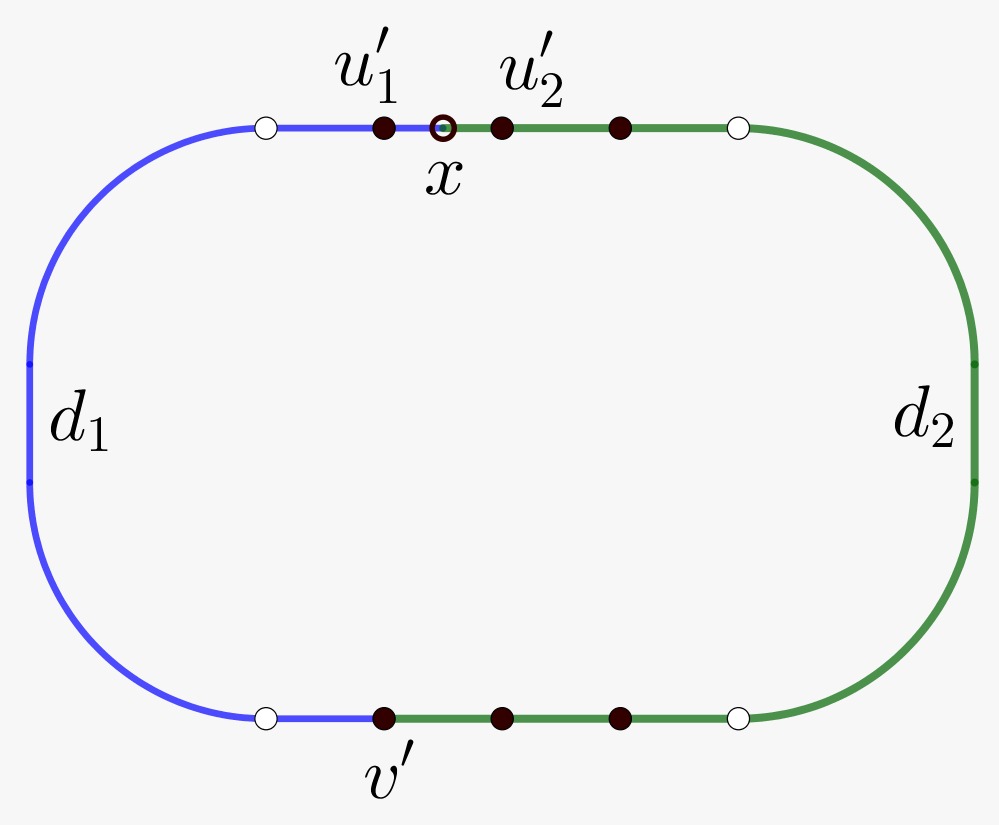}
\caption{If $u'_1, \, u'_2, \, v' \in J'(G)$ and  $d(u'_1,u'_2)=1/4$ then $d(u'_1,v') \not = d(u'_2,v')$. }
\label{fig:23}
\end{figure}

\begin{proposition}\label{Th:1punto}
Let $G$ be a graph and $x,y,z,w \in G$  such that $\dx=(x,y)_w-(x,z)_w$.  If $x_1,x_2,x_3 \in \{ x,y,z,w \} \cap J'(G)$ then there exists $x_4 \in J'(G)$ such that $\delta^* (x_1,x_2,x_3,x_4) \geq \delta^* (x,y,z,w)$.
\end{proposition}
\begin{proof}
%Consider $\dx = (x,y)_w-(x,z)_w$. 
Consider the following four cases: \\
\emph{Case i).} %If $x \in J'(G)$ we finish, so, we suppose that $x \not  \in J'(G)$
 $x \not  \in J'(G)$. Consider $x \in v^x_0v^x_1 \in E(G)$ with $v^x_0 \in [xy]$ (see Figure \ref{fig:1punto}).  We denoted by $v'_x$ and $v''_x$ the ends of the quarter of the edge on which $x$ is situated with $d(v^x_0,v'_x) <d(v^x_0,x)<d(v^x_0,v''_x)$ . % los extremos de la cuarta parte de la arista en la que se encuentra $x$ con $d(v_0,v'_x) <d(v_0,x)<d(v_0,v''_x)$.
%Tomemos $x'=v'_x$. Tenemos los siguientes casos:
If $v'_x \in [xz]$ we take $x'=v''_x$; otherwise, we take $x'=v'_x$. By Proposition \ref{Th:Pcuarto-Pgen} we have $d(x',z)=d(x,z)+d(x',x)$, $d(x',y) \leq d(x,y)+d(x,x')$, $d(x',w) \leq d(x,w)+d(x,x')$. So, we have 
\begin{align*}
(x',y)_w-(x',z)_w &= \frac{1}{2} \Bigl( d(x',w)+ d(y,w)-d(x',y)-d(x',w)-d(z,w)+d(x',z) \Bigr) \\
&\geq  \frac{1}{2} \Bigl( d(y,w)-\bigl( d(x,y)+d(x,x') \bigr) -d(z,w)+ \bigl( d(x,z)+d(x,x') \bigr) \Bigr) \\
 &=(x,y)_w-(x,z)_w .\\
 (y,z)_w-(x',z)_w &= \frac{1}{2} \Bigl( d(y,w)+d(z,w)-d(y,z)-d(x',w)-d(z,w)+d(x',z) \Bigr) \\
 & \geq \frac{1}{2} \Bigl(  d(y,w)-d(y,z)-\bigl( d(x,w)+d(x,x') \bigr)+ \bigl( d(x,z) +d(x,x') \bigr) \Bigr) \\
 &=  (y,z)_w-(x,z)_w \\
 &\geq (x,y)_w-(x,z)_w .
\end{align*}
So, we have
$$ \delta^*(x',y,z,w) \geq \min \{ (x',y)_w,(y,z)_w\}-(x',z)_w \geq (x,y)_w-(x,z)_w = \delta^*(x,y,z,w).$$

\medskip

\emph{Case ii).} 
If $z \not \in J'(G)$. Similarly to the previous case, we consider $z \in v^z_0v^z_1 \in E(G)$ with $v^z_0 \in [zw]$.  We denoted by $v'_z$ and $v''_z$ the ends of the quarter of the edge on which $z$ is situated with $d(v^z_0,v'_z) <d(v^z_0,z)<d(v^z_0,v''_z)$.  
If $v'_z \in [xz]$ we take $z'=v''_z$; otherwise, we take $z'=v'_z$. By Proposition \ref{Th:Pcuarto-Pgen} we have $d(x,z')=d(x,z)+d(z,z')$, $d(z',y) \leq d(z,y)+d(z,z')$, $d(z',w) \leq d(z,w)+d(z,z')$. So, we have
\begin{align*}
(x,y)_w-(x,z')_w &= \frac{1}{2} \left( d(x,w)+ d(y,w)-d(x,y)-d(x,w)-d(z',w)+d(x,z') \right) \\
&\geq  \frac{1}{2} \left( d(y,w)-d(x,y) -\bigl( d(w,z)+d(z,z') \bigr) + \bigl( d(x,z)+d(z,z') \bigr) \right) \\
 &=(x,y)_w-(x,z)_w .\\
 (y,z')_w-(x,z')_w &= \frac{1}{2} \left( d(y,w)+d(z',w)-d(y,z')-d(x,w)-d(z',w)+d(x,z') \right) \\
 & \geq \frac{1}{2} \left(  d(y,w)-\bigl( d(y,z)+ d(z,z') \bigr)- d(x,w) + \bigl( d(x,z) +d(z,z') \bigr) \right) \\
 &=  (y,z)_w-(x,z)_w \\
 &\geq (x,y)_w-(x,z)_w .
\end{align*}
So, we have
$$ \delta^*(x,y,z',w) \geq \min \{ (x,y)_w,(y,z')_w\}-(x,z')_w \geq (x,y)_w-(x,z)_w = \delta^*(x,y,z,w).$$

\medskip

\emph{Case iii).} 
$y \not \in J'(G)$. Consider $y \in v^y_0v^y_1 \in E(G)$ with $v^y_0 \in [xy]$.  We denoted by $v'_y$ and $v''_y$ the ends of the quarter of the edge on which $y$ is situated with $d(v^y_0,v'_y) <d(v^y_0,y)<d(v^y_0,v''_y)$.  
If $v'_y \in [yw]$ we take $y'=v''_y$; otherwise, we take $y'=v'_y$. By Proposition \ref{Th:Pcuarto-Pgen} we have $d(y',w)=d(y,w)+d(y,y')$, $d(x,y') \leq d(x,y)+d(y,y')$, $d(y',z) \leq d(y,z)+d(y,y')$. Following a similar procedure, we obtain 
\begin{align*}
(x,y')_w-(x,z)_w & \geq  \frac{1}{2} \left( \left( d(y,w)+ d(y,y') \right)  - \left( d(x,y) +d(y,y') \right) - d(w,z) + d(x,z) \right) \\
 &=(x,y)_w-(x,z)_w .\\
 (y',z)_w-(x,z)_w & \geq \frac{1}{2} \left(  \left(  d(y,w) + d(y,y') \right)  -\left( d(y,z) + d(y,y') \right)- d(x,w) + d(x,z)  \right) \\
 &=  (y,z)_w-(x,z)_w \\
 &\geq (x,y)_w-(x,z)_w .
\end{align*}
So, we have
$$ \delta^*(x,y',z,w) \geq \min \{ (x,y')_w,(y',z)_w\}-(x,z)_w \geq (x,y)_w-(x,z)_w = \delta^*(x,y,z,w).$$

\medskip

\emph{Case iv).} 
 $w \not \in J'(G)$. Consider $w \in v^w_0v^w_1 \in E(G)$ with $v^w_0 \in [wz]$.  We denoted by $v'_w$ and $v''_w$ the ends of the quarter of the edge on which $w$ is situated with $d(v^w_0,v'_w) <d(v^w_0,w)<d(v^w_0,v''_w)$.  
If $v'_w \in [yw]$ we take $w'=v''_w$; otherwise, we take $w'=v'_w$. By Proposition \ref{Th:Pcuarto-Pgen} we have $d(y,w')=d(y,w)+d(w,w')$, $d(x,w') \leq d(x,w)+d(w,w')$, $d(w',z) \leq d(w,z)+d(w,w')$. Following a similar procedure, we obtain 
\begin{align*}
(x,y)_{w'}-(x,z)_{w'} & \geq  \frac{1}{2} \left( \left( d(y,w)+ d(w,w') \right)  - d(x,y) -\left( d(z,w) +d(w,w') \right) + d(x,z) \right) \\
 &=(x,y)_w-(x,z)_w .\\
 (y,z)_{w'}-(x,z)_{w'} & \geq \frac{1}{2} \left(  \left(  d(y,w) + d(w,w') \right)  - d(y,z) -\left( d(x,w) + d(w,w') \right) + d(x,z)  \right) \\
 &=  (y,z)_w-(x,z)_w \\
 &\geq (x,y)_w-(x,z)_w .
\end{align*}
So, we have
$$ \delta^*(x,y,z,w') \geq \min \{ (x,y)_{w'},(y,z)_{w'}\}-(x,z)_{w'} \geq (x,y)_w-(x,z)_w = \delta^*(x,y,z,w).$$
\end{proof}

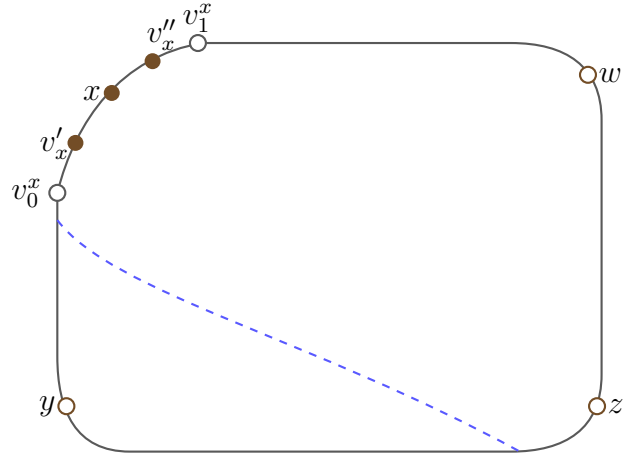
\begin{figure}[ht]
\centering

\begin{tikzpicture}[scale=1.2]
\centering
% Main rounded rectangular curve %dibujamos apartir de v_0^x en sentido de las manecillas del reloj
\draw[thick, gray!70!black]
  (-0.2,0.6)
  .. controls (0,1.4) and (0.5,2.1) .. (1.4,2.25)
  -- (4.8,2.25)
  .. controls (5.4,2.25) and (5.8,2.0) .. (5.8,1.4)
  -- (5.8,-1.4)
  .. controls (5.8,-2.0) and (5.4,-2.25) .. (4.8,-2.25)
  -- (0.6,-2.25)
  .. controls (0.0,-2.25) and (-0.2,-1.8) .. (-0.2,-1.2)
  -- (-0.2,0.6);

% Dashed blue chord
\draw[dashed, blue!65, thick]
  (-0.2,0.3)
  .. controls (0.4,-0.5) and (2.6,-1.0) .. (4.9,-2.25);

% Marked points
\fill[white] (-0.2,0.6) circle (2.5pt);
\draw[gray!70!black, thick] (-0.2,0.6) circle (2.5pt);

%vertices v'_x, x, v_x''
\fill[brown!60!black] (0.0,1.15) circle (2.5pt);
\fill[brown!60!black] (0.4,1.7) circle (2.5pt);
\fill[brown!60!black] (0.85,2.05) circle (2.5pt);

%vertic v_10†0^x
\fill[white] (1.35,2.25) circle (2.5pt);
\draw[gray!70!black, thick] (1.35,2.25) circle (2.5pt);

%vertice w
\fill[white] (5.65,1.9) circle (2.5pt);
\draw[brown!60!black, thick] (5.65,1.9) circle (2.5pt);

%vertice z
\fill[white] (5.75,-1.75) circle (2.5pt);
\draw[brown!60!black, thick] (5.75,-1.75) circle (2.5pt);

%vertice y
\fill[white] (-0.1,-1.75) circle (2.5pt);
\draw[brown!60!black, thick] (-0.1,-1.75) circle (2.5pt);

% Labels
\node[left] at (-0.25,0.6) {$v_0^x$};
\node[above] at (1.35,2.25) {$v_1^x$};

\node[left] at (0.05,1.15) {$v_x'$};
\node[left] at (0.4,1.7) {$x$};
\node[above] at (0.95,2.05) {$v_x''$};

\node[right] at (5.65,1.9) {$w$};
\node[right] at (5.75,-1.75) {$z$};
\node[left] at (-0.1,-1.75) {$y$};

\end{tikzpicture}

\caption{By taking $x'$ outside of $[xz]$ in this way, we obtain that $\delta'(x',y,z,w) \geq \delta'(x,y,z,w)$. }
\label{fig:1punto}

\end{figure}

Given a graph $G$ and $x,y,z,w \in G$ such that $\delta^*(x,y,z,w) = (x,y)_w - (x,z)_w$, we analyze the situations where exactly two points from the set $\{x,y,z,w\}$ do not belong to $J'(G)$. Specifically, the next result focuses on the pairs $\{x,y\}$, $\{x,w\}$, $\{y,z\}$, and $\{w,z\}$; the remaining cases, namely $\{x,z\}$ and $\{y,w\}$, will be addressed later (see Proposition \ref{Th:2puntos2}).

%%%%%%%%%%%%%%%%%%%%%%%%%%%%%%%%%%%%%%%%%%%%%%%%%%%%%%%%%%%%%%%%%%%%%%%%%%%%%%%%%%%%%%%%%%%%%%%%%%%%%%%%%%%%%%%%%%%%%%%%%%%%%%
\begin{proposition}\label{Th:2puntos1} %Teorema xy
Let $G$ be a graph and $x,y,z,w \in G$  such that $\dx=(x,y)_w-(x,z)_w$, then: \\
If  $z,w \in J^{\prime}(G)$, then there exist $x^{\prime}, y^{\prime} \in J^{\prime}$ such that $\d ^*(x',y',z,w) \geq \d^*(x,y,z,w)$. \\
If $x, y \in J'(G)$, then there exist $z', w' \in J'(G)$ such that $\d ^*(x,y,z',w') \geq \d^*(x,y,z,w)$. \\
If  $y,z \in J^{\prime}(G)$, then there exist $x^{\prime}, w^{\prime} \in J^{\prime}$ such that $\d ^*(x',y,z,w') \geq \d^*(x,y,z,w)$. \\
If $x, w  \in J'(G)$, then there exist $y', z' \in J'(G)$ such that $\d ^*(x,y',z',w) \geq \d^*(x,y,z,w)$. 
\end{proposition}

\begin{proof}
First, let us consider $z,w \in J'(G)$. Let us assume that $x, y \not \in J'(G)$. 
Note that it suffices to prove  that there exists $x\pr \in J\pr$ such that $\delta^*(x\pr,y,z,w) \geq \delta^*(x,y,z,w)$ and so Proposition  \ref{Th:1punto} gives the result.
Since $x\notin J\pr $ then we assume that $x \in v_0v_1 \in E(G)$.  %with $v_0 \in [xy]$. 
Consider $v'_x, \, v''_x \in J'(v_0v_1)$  such that  $d(v_0,v'_x) < d(v_0,x)< d(v_0,v''_x)$ and   $d(v'_x, \, v''_x)=1/4$. We assume $v_0, \, v'_x \in [xy]$.
If $v'_x \in [xz]$ we take $x'=v''_x$; otherwise, we take $x'=v'_x$. By Proposition \ref{Th:Pcuarto-Pgen} we have $d(x',z)=d(x,z)+d(x',x)$, $d(x',y) \leq d(x,y)+d(x,x')$, $d(x',w) \leq d(x,w)+d(x,x')$. So, we have 
\begin{align*}
(x',y)_w-(x',z)_w &= \frac{1}{2} \left( d(x',w)+ d(y,w)-d(x',y)-d(x',w)-d(z,w)+d(x',z) \right) \\
&\geq  \frac{1}{2} \left( d(y,w)-\left( d(x,y)+d(x,x') \right) -d(z,w)+ \left( d(x,z)+d(x,x') \right) \right) \\
 &=(x,y)_w-(x,z)_w .\\
 (y,z)_w-(x',z)_w &= \frac{1}{2} \left( d(y,w)+d(z,w)-d(y,z)-d(x',w)-d(z,w)+d(x',z) \right) \\
 & \geq \frac{1}{2} \left(  d(y,w)-d(y,z)-\left( d(x,w)+d(x,x') \right)+ \left( d(x,z) +d(x,x') \right) \right) \\
 &=  (y,z)_w-(x,z)_w \\
 &\geq (x,y)_w-(x,z)_w .
\end{align*}
So, we have
$$ \delta^*(x',y,z,w) \geq \min \{ (x',y)_w,(y,z)_w\}-(x',z)_w \geq (x,y)_w-(x,z)_w = \delta^*(x,y,z,w).$$
\smallskip
If $y,z \in J^{\prime}$, we proceed exactly as in the previous case.
\smallskip

If $x,y \in J'(G)$ or $x,w \in J'(G)$ we have $z \not \in J'(G)$. We proceed exactly as in \emph{case ii} of Proposition \ref{Th:1punto}, obtaining $z' \in J'(G)$ such that $\delta^*(x,y,z',w) \geq \delta^*(x,y,z,w)$, and proposition \ref{Th:1punto} gives the result.
\end{proof}

\begin{proposition}\label{Prop:ww'}
Let $G$ be a graph and  $x,y,z,w \in G$. Consider $\dx= (x,y)_w-(x,z)_w$ and $d(y,w)=k+i/4$, $k \in \NN$, $i\in \{1,3 \}$, then there exists $w' \in G$ (respectively $y' \in G$) such that $d(w,w')=1/4$ (respectively $d(y,y')=1/4$) and  $\delta^*(x,y,z,w') \geq \dx$ (respectively $\delta^*(x,y',z,w) \geq \dx$).
\end{proposition}
\begin{proof}
Let $w_x, \, w_z \in G$ such that $w_x \in [xw]$, $w_z \in [zw]$ and $d(w,w_x)=d(w,w_z)=1/4$. We have the following cases: \\
%Consider $w_x \in [xw]$ and $w_z \in [zw]$ such that $d(w,w_x)=d(w,w_z)=1/4$. We have the following cases: \\
\emph{Case $1$}.  $[xw] \cap [yw] \not = \{ w \}$ (see Figure \ref{fig:ww'}). Taking $w' =w_z$, we have $d(x,w') \leq d(x,w)+1/4$, $d(z,w')=d(z,w)-1/4$ and $d(y,w')=d(y,w)+1/4$. So, we have
%Suppose that
\begin{align*}
(x,y)_{w'}-(x,z)_{w'}&= \frac{1}{2} \left( d(x,w')+ d(y,w')-d(x,y)-d(x,w')-d(z,w')+d(x,z) \right) \\
&= \frac{1}{2} \left( d(y,w)+1/4-d(x,y)-(d(z,w)-1/4) +d(x,z)    \right) \\
&= (x,y)_{w}-(x,z)_{w}+1/4. \\
(y,z)_{w'}-(x,z)_{w'}&= \frac{1}{2} \left( d(y,w')+d(z,w')-d(y,z)-d(x,w')-d(z,w')+d(x,z)  \right) \\
&\geq \frac{1}{2} \left( d(y,w) +1/4 - d(y,z)-(d(x,w)+1/4)+d(x,z) \right) \\
%&=\frac{1}{2} \left( d(y,w)+d(z,w)-d(y,z)-d(x,w)-d(z,w)+d(x,z)  \right) \\
&=(y,z)_{w}-(x,z)_{w} \\
&\geq (x,y)_{w}-(x,z)_{w}.
\end{align*}
\emph{Case $2$}.  $[zw] \cap [yw] \not = \{ w \}$. Consider $w' =w_x$. We have $d(x,w') = d(x,w)-1/4$, $d(z,w')\leq d(z,w)+1/4$ and $d(y,w')=d(y,w)+1/4$. So, we have
\begin{align*}
(x,y)_{w'}-(x,z)_{w'}&= \frac{1}{2} \left( d(x,w')+ d(y,w')-d(x,y)-d(x,w')-d(z,w')+d(x,z) \right) \\
&\geq \frac{1}{2} \left( d(y,w)+1/4-d(x,y)-(d(z,w)+1/4) +d(x,z)    \right) \\
&= (x,y)_{w}-(x,z)_{w}. \\
(y,z)_{w'}-(x,z)_{w'}&= \frac{1}{2} \left( d(y,w')+d(z,w')-d(y,z)-d(x,w')-d(z,w')+d(x,z)  \right) \\
&\geq \frac{1}{2} \left( d(y,w) +1/4 - d(y,z)-(d(x,w)-1/4)+d(x,z) \right) \\
%&=\frac{1}{2} \left( d(y,w)+d(z,w)-d(y,z)-d(x,w)-d(z,w)+d(x,z)  \right) \\
&=(y,z)_{w}-(x,z)_{w} +1/4 \\
&\geq (x,y)_{w}-(x,z)_{w}+1/4.
\end{align*}
In both cases we have that 
$\min \{(x,y)_{w'},(y,z)_{w'} \} -((x,z)_{w'} \geq (x,y)_w-(x,z)_w$. Therefore  
$$\delta^*(x,y,z,w') \geq \dx. $$
The respective proof for $y$ and $y'$ is similar.
\end{proof}

\begin{remark}\label{Rem:ww'}
Note that, in proposition \ref{Prop:ww'}, if $(y,z)_w=(x,y)_w$,  then $ \vert (y,z)_{w'}-(x,y)_{w'} \vert =1/4$.
\end{remark}

\begin{figure}[htbp]
\centering
\begin{tikzpicture}[
    scale=1.6, 
    % Estilo para las l¨ªneas normales del grafo (negras y limpias)
    graph edge/.style={draw=black!85, line width=0.9pt, line cap=round, line join=round},
    % Estilo para la diagonal destacada (roja)
    geodesic/.style={draw=red!75!black, line width=1.1pt, line cap=round},
    % Estilos para los v¨¦rtices
    vertex/.style={circle, fill=black, inner sep=0pt, minimum size=4.5pt}
]

    % 1. Definici¨®n geom¨¦trica de las coordenadas (Rect¨¢ngulo base)
    \coordinate (y)   at (0,0);
    \coordinate (wz)  at (5,2);
    \coordinate (z)   at (5,0);
    \coordinate (w)   at (5,2.5); % y est¨¢ exactamente sobre z (misma X)
    \coordinate (corner) at (0,2.5); % Esquina superior izquierda te¨®rica
    
    % Puntos sobre el lado izquierdo del rect¨¢ngulo
    \coordinate (wx)  at (4.5,2.5);  
    \coordinate (x)   at (0,2.5);

    % 2. Trazado de las aristas del contorno (L¨ªneas negras)
    \path[graph edge] (wz) -- (z);
    \path[graph edge] (z)  -- (y);
    \path[graph edge] (y)  -- (corner) -- (x); % Sigue el borde superior recto
    \path[graph edge] (x)  -- (wx);
    \path[graph edge] (wx) -- (w);
    \path[graph edge] (w)  -- (wz);

    % 3. Trazado de la geod¨¦sica interna (L¨ªnea roja diagonal)
    \path[geodesic] (y) -- (wx);

    % 4. Dibujo de los puntos (Nodos sobre las l¨ªneas)
    \node[vertex] at (x)  {};
    \node[vertex] at (wx) {};
    \node[vertex] at (w)  {};
    \node[vertex] at (wz) {};
    \node[vertex] at (z)  {};
    \node[vertex] at (y)  {};

    % 5. Etiquetas de texto bien posicionadas
    \node[left]        at (x)  {$x$};
    \node[above]        at (wx) {$w_x$};
    \node[right]  at (w)  {$w$};
    \node[right]       at (wz) {$w_z$};
    \node[below right] at (z)  {$z$};
    \node[below left] at (y)  {$y$};
%    \node[above right] at (y)  {$y$};

\end{tikzpicture}
\caption{If $d(y,w)=k+i/4$, $k \in \mathbb{N}$, $i \in \{1,3 \}$, exists $w'$ such that $\delta^*(x,y,z,w') \geq \delta^*(x,y,z,w)$. }
\label{fig:ww'}
\end{figure}
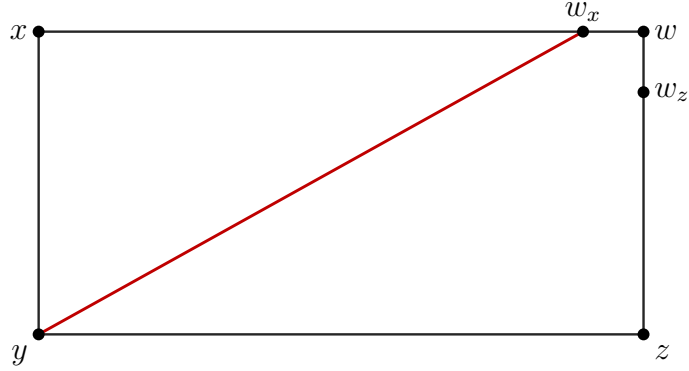

The proof of the following result follows a similar way to the proof of Proposition \ref{Prop:ww'}.

\begin{lemma} \label{Prop:xx'}
Let $G$ be a graph and  $x,y,z,w \in G$. Consider $\dx=(x,y)_w-(x,z)_w$.
%Let $G$ be a graph and  $x,y,z,w \in G$. Consider $\dx= \min\{ (x,y)_w,\, (y,z)_w \} -(x,z)_w=(x,y)_w-(x,z)_w$.
 If $(y,z)_w-(x,y)_w = \alpha >0$ and $x' \in [xy]$, $z' \in [zw]$ with $d(x,x') \leq \alpha$ or $d(z,z') \leq \alpha$, then $\delta^*(x',y,z,w) = \dx$ or  $\delta^*(x,y,z',w) = \dx$.
\end{lemma}

\begin{remark} \label{Rem:xx'}
In Lemma \ref{Prop:xx'}, if $d(x,x') = \alpha$, $d(z,z') = \alpha$ or $d(x,x') =d(z,z') = \alpha/2$,  then $\vert (y,z)_w-(x',z)_w \vert=0$, $\vert (y,z')_w-(x,z')_w \vert=0$ or  $\vert (y,z')_w-(x',z')_w \vert=0$ respectively.
\end{remark}

\begin{lemma}\label{Th:2pts-2pts}
%\begin{theorem}\label{Th:2pts-2pts}
Let $G$ be a graph and $x,z \in G \setminus J'(G)$. Consider $v'_x,v'_z,v''_x,v''_z \in J'(G)$ such that $x \in [v'_x v''_x]$, $z \in [v'_zv''_z]$ and $v''_x,v''_z \in [xz]$. 
If $d(v'_x,v'_z)< d(x,z)$ then 
$$
d(v'_x,v'_z)=d(v''_x,v''_z),  \quad \quad  d(x,v''_x)+d(z,v''_z)<1/4, \quad d(x,y)<d(v'_x,v''_z)=d(v'_z,v''_x).
$$
 In addition, if $ d(x,z) <d(v'_x,v'_z)$ then
$$ d(v'_x,v'_z) + d(v''_x,v''_z) = k/2, $$
$$ d(v'_x,v'_z) \geq d(v''_x,v''_z) + l/2, $$
for $k, \,l \in \mathbb{N}$, $k$ odd. 
 
%\end{theorem}
\end{lemma}

\begin{figure}[h!]
\centering
\begin{subfigure}[b]{0.4\linewidth}
\includegraphics[width=\linewidth]{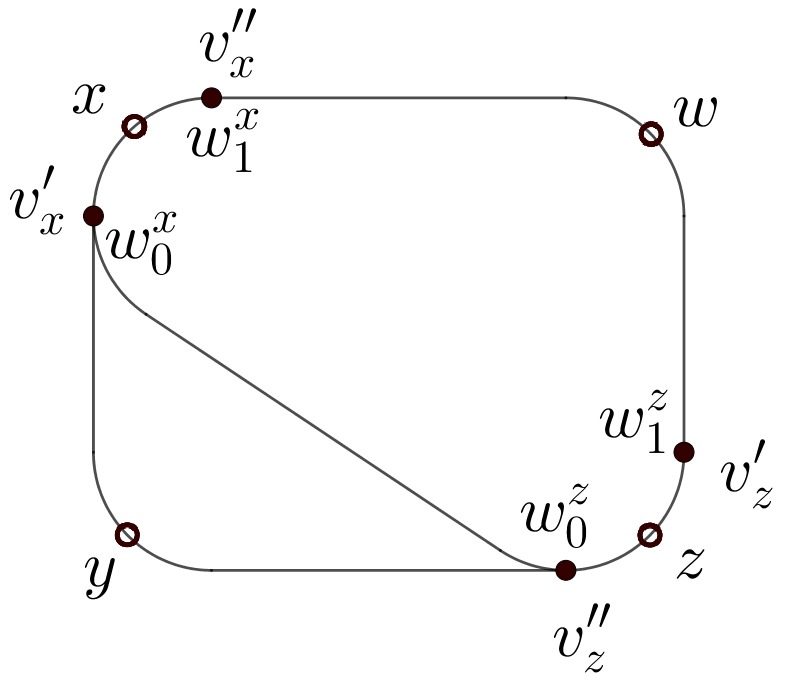}
\caption{Case a.}
\label{fig:xz_a}
\end{subfigure}
\begin{subfigure}[b]{0.4\linewidth}
\includegraphics[width=\linewidth]{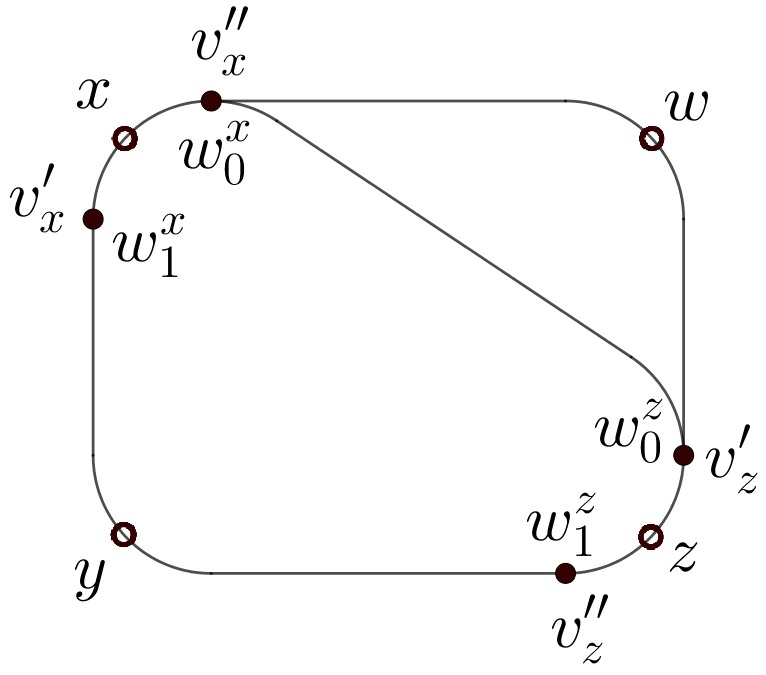}
\caption{Case b.}
\label{fig:xz_b}
\end{subfigure}
\begin{subfigure}[b]{0.4\linewidth}
\includegraphics[width=\linewidth]{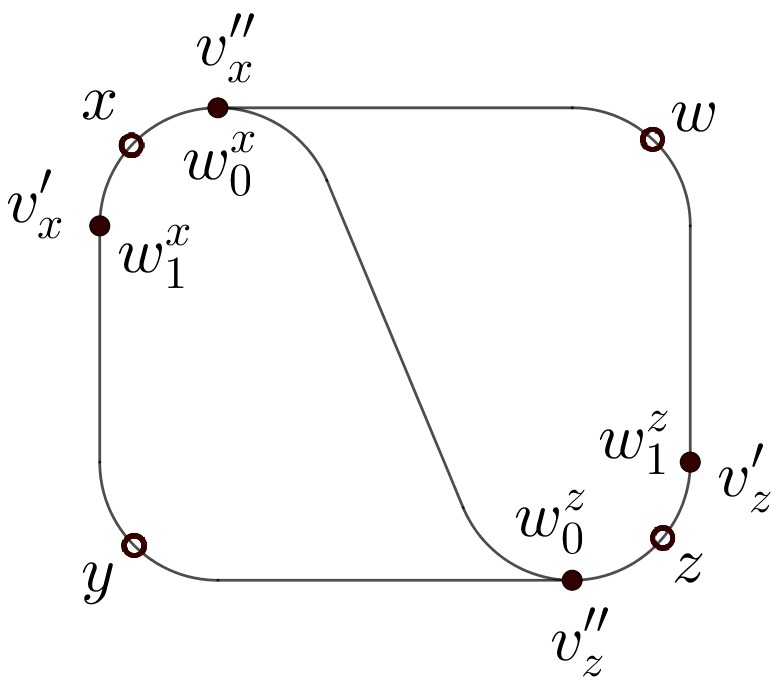}
\caption{Case c.}
\label{fig:xz_c}
\end{subfigure}
\begin{subfigure}[b]{0.4\linewidth}
\includegraphics[width=\linewidth]{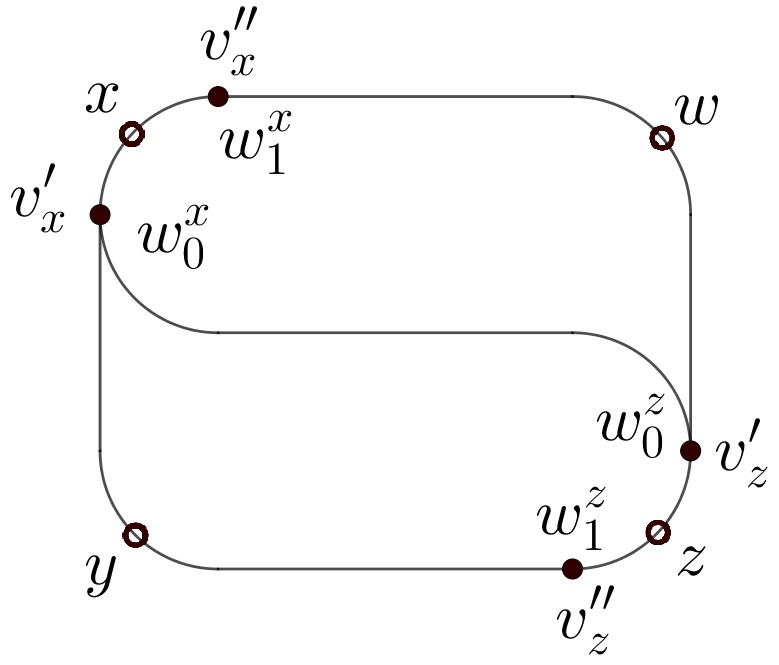}
\caption{Case d.}
\label{fig:xz_d}
\end{subfigure}
\caption{Possible cases for geodesics $[xz]$.}
\label{fig:xz}
\end{figure}

%%%%%%%%%%%%%%%%%%%%%%%%%%%%%%%%%%%%%%%%%%XZ   XZ   XZ
%teorema xz
\begin{proposition}\label{Th:2puntos2} %x,z
Let $G$ be a graph and $x,y,z,w \in G$  such that $\delta^*(x,y,z,w)=(x,y)_w-(x,z)_w$.  
If $y,w \in J'(G)$ and $x,z \not \in J'(G)$, or  $x,z \in J'(G)$ and  $y,w \not \in J'(G)$, then there exist $x', \,  y', \,  z', \, w' \in J'(G)$ such that  
$$
\d^*(x',y',z',w') \geq \d^*(x,y,z,w).
$$
\end{proposition}

\begin{proof}
Consider $y,w \in J'(G)$ and $x,z \not \in J'(G)$. 
Suppose that $x \in v_0^xv_1^x$ and  $z \in v_0^zv_1^z$ with $v_0^xv_1^x, \, v_0^zv_1^z \in E(G)$.  We denoted by $v'_xv''_x, \, v'_zv''_z$  the segments that correspond to the quarter part of edges where  $x$ and $z$  are contained, such that $d(v^x_0,v'_x)<d(v^x_0,x)<d(v^x_0,v''_x)$ and  $d(v^z_0,v'_z)<d(v^z_0,z)<d(v^z_0,v''_z)$. % with  $v'_x,v''_x.v'_z,v''_z \in J'(G)$. 
Consider $v_0^x \in [xy]$ and $v_0^z \in [wz]$. If $v_0^x \in [xw]$ Proposition \ref{Prop:vinxyxw} gives $\d^*(v_0^x,y,z,w) \geq \d^*(x,y,z,w)$ and by Proposition \ref{Th:2puntos1} exists $z' \in  J'(G)$ such that $\d^*(v_0^x,y,z',w) \geq \dx$, therefore we can assume that $v_1^x \in [xw]$. Similarly, we have the result when $v_0^z \in [yz]$. Hence assume that  $v_1^z \in [yz]$. 

There are four possibilities for the geodesic $[xz]$ (see Figure \ref{fig:xz}). Consider $w_0^x,w_1^x \in \{v_x', v_x'' \}$ and  $w_0^z,w_1^z \in \{v_z', v_z'' \}$ such that $w_0^x,w_0^z \in [xz]$. 
 
\emph{Case 1.}  $d(x,z)<d(w_1^x,w_1^z)$. 

By Lemma \ref{Th:2pts-2pts} $[w_1^xx]\cup[xz] \cup [zw_1^z]$ is a geodesic. Taking $x'=w_1^x$ and $z'=w_1^z$. We have 
$d(x',z')=d(x',x)+d(x,z)+d(z,z')$ and $d(x',y) \leq d(x,y)+d(x,x')$, $d(z',w) \leq d(z,w)+d(z,z')$, $d(y,z') \leq d(y,z)+d(z,z')$, $d(x',w) \leq d(x,w)+d(x,x')$. We have
\begin{align*}
(x',y)_w-(x',z')_w 
&= \frac{1}{2} \left( d(x',w)+d(w,y)-d(x',y)-d(x',w)-d(w,z')+d(x',z') \right) \\
&= \frac{1}{2} \left( d(w,y)-d(x',y)-d(w,z')+d(x',z') \right) \\
&\geq \frac{1}{2} \left( d(w,y)-\left( d(x,y) +d(x,x') \right) -\left( d(w,z) +d(z,z') \right) \right. \\
&\qquad \left. + \, \left(d(x,x') + d(x,z) +d(z,z') \right) \right) \\
&= \frac{1}{2} \left( d(y,w)-d(x,y)- d(z,w) +d(x,z) \right) \\
&= \frac{1}{2} \left( d(x,w)+d(w,y)-d(x,y)-d(x,w)-d(w,z)+d(x,z) \right) \\
&= (x,y)_w-(x,z)_w. \\[1ex] % A0Š9ade un peque0Š9o espacio extra antes del segundo bloque
(y,z')_w-(x',z')_w 
&= \frac{1}{2} \left( d(y,w)-d(y,z')-d(x',w)+d(x',z') \right) \\
&= \frac{1}{2} \left( d(y,w)-\left( d(y,z) + d(z,z') \right) - \left( d(x,w) +d(x,x') \right) \right. \\
&\qquad \left. + \, \left(d(x',x)+ d(x,z) +d(z,z') \right) \right) \\
&= \frac{1}{2} \left( d(y,w)-d(y,z)-d(x,w)+d(x,z) \right) \\
&= (y,z)_w-(x,z)_w \\
&> (x,y)_w-(x,z)_w.
\end{align*}

\emph{Case 2.}  $d(w_1^x,w_1^z)<d(x,z)$. 

In this case there exists a cycle $C$ of length $L(C)$. We claim  that    $d(x,z)= L(C)/2$ otherwise there exist $x^*, z^*$ with $x^* \in xw^x_1$ and $z^* \in zw^z_1$ such that $d(x^*,z^*)=L(C)/2$ and $\delta^*(x^*,y,z^*,w) \geq \delta^*(x,y,z,w)$. We have that $d(v'_x,v'_z)=d(x,z)=d(v''_x,v''_z)$ or $d(v'_x,v''_z)=d(x,z)=d(v''_x,v'_z)$ \emph{i.e.} cases $a$ and $b$ in  Figure \ref{fig:xz} coincide, in the same way cases $c$ and $d$ in  Figure \ref{fig:xz}.

Consider $\alpha= \vert (y,z)_w-(x,y)_w \vert$. If $\alpha \geq 1/4$ taking $x'=v'_x$ Lemma \ref{Prop:xx'} and Remark \ref{Rem:xx'} give 
$$\delta^*(x',y,z,w) \geq \min \{(x',y)_w,(y,z)_w \} -(x',z)_w = \delta^*(x,y,z,w),$$
so, by Proposition \ref{Th:1punto} there exists $z' \in J'(G)$ such that $\delta^*(x',y,z',w) \geq \delta^*(x,y,z,w)$.

If $\alpha=0$ we have $(y,z)_w=(x,y)_w$, \emph{i.e.} $d(y,v'_x)+d(v'_x,x)+ d(w,v'_z)+d(v'_z,z) =
d(x,v''_x) +d(v''_x,w)+ (d(y,v''_z)+d(v''_z,z)$. So  
\begin{equation} \label{eq:medmen}
\left \vert \left(  d(y,v'_x)+ d(w,v'_z) \right) -\left( d(v''_x,w)+ (d(y,v''_z)  \right) \right \vert = 2\left \vert d(x,v''_x)-d(x,v'_x) \right \vert .
\end{equation}
Since $d(y,v'_x)+ d(w,v'_z)$ and $d(v''_x,w)+ d(y,v''_z)$ 
 are multiples of $1/4$ their difference must also be a multiple of $1/4$, furthermore  $ \left \vert d(x,v''_x)-d(x,v'_x) \right \vert <1/4$ so equality  \ref{eq:medmen} must be $0$ o $1/4$. 
 Consider 
$d(y,v'_x)=k_1$, $ d(w,v'_z)=k_2$, $d(v''_x,w)=k_3$, $d(y,v''_z)=k_4$, $k_1+k_2=k_5$, $k_3+k_4=k_6$.  
Since $k_1+k_2+k_3+k_4+1/2$ 
must be integer, we have that $k_1+k_2+k_3+k_4$ 
must be an odd multiple of $1/2$, 
so $k_5$ and $k_6$ are both odd multiplies of $1/4$ %ya que si fueran ambos multiplos pares de 1/4 susuma seria un entero
so equality \ref{eq:medmen} must be $0$, \emph{i.e.}, $d(x,v'_x)=d(x,v''_x)=d(z,v'_z)=d(z,v''_z)=1/8$.  
Since $k_5$ and $k_6$ are odd multiples of $1/4$ then if $k_1$ is an even multiply of $1/4$ then $k_2$ must be an odd multiply. Similarly $k_3$ and $k_4$ also have different parity.
Note that if $v'_x \in J(G)$ (respectively $J'(G)\setminus J(G)$) then  $v''_x, \, v''_z \in J'(G) \setminus J(G)$ (respectively $J(G)$) and  $v'_z \in J(G)$ (respectively $J'(G) \setminus J(G)$), 
therefore $k_1$ and $k_3$ are odd and $k_2$ and $k_4$ are even or vice versa.
 Therefore if $y \in J(G)$ then $w \in J'(G) \setminus J(G)$ or whether $y \in J'(G) \setminus J(G)$ then $w \in J(G)$. 
Since we proceed in a similar way in both cases, let us assume that  $w \in  J'(G) \setminus J(G)$, By Proposicion \ref{Prop:ww'} and  Remark \ref{Rem:ww'} there exists $w' \in J(G)$ such that  $d(w,w')=1/4$ and $\delta^*(x,y,z,w') \geq \delta^*(x,y,z,w)$ and $(y,z)_{w'} - (x,y)_{w'}=1/4$. Taking $x'=v'_x$, $z'=v'_z$, we have $(x',z')_{w'}=(y,z')_{w'}$. Lemma \ref{Prop:xx'} and Remark  \ref{Rem:xx'} give  
$$
\delta^*(x',y,z',w') \geq \delta^*(x,y,z,w).
$$

If $0<\alpha<1/4$. We have two subcases. 

%%%%%%%%%%%%%%%%%%%%%%%%%%%%%%%%%%%%%%%%%%%%%%%%%%%%%%%%%%%%%%%%%%%%%%%%%%%%%%%%%%%%%%%%%%%%%%%%%%%%%%%%%
% NOTA (cambio marcado en rojo): el argumento de contradiccion original para el Caso i (0<alpha<1/4)
% dividia el desplazamiento entre x y z en alpha/2 cada uno; se reemplaza por una prueba constructiva
% directa que mueve unicamente x en alpha y reduce al caso alpha=0 ya probado. Se conserva la version
% anterior comentada abajo para referencia del equipo editor.
%
% \emph{Case i).} $d(v'_x,v'_z)=d(x,z)$ (case $a$ or $b$ in Figure \ref{fig:xz}).
% We take $x' \in [xy], \, z' \in [zw]$ such that  $d(x,x')=d(z,z')=\alpha/2$.  Lemma \ref{Prop:xx'} and Remark  \ref{Rem:xx'} give $\delta^*(x',y,z',w) \geq \min \{ (x',y)_w, (y,z')_w \} -(x',z')_w=(x,y)_w-(x,z)_w+ \alpha/2 > \delta^*(G)$
%  which is a contradiction, so this case is not possible.
%%%%%%%%%%%%%%%%%%%%%%%%%%%%%%%%%%%%%%%%%%%%%%%%%%%%%%%%%%%%%%%%%%%%%%%%%%%%%%%%%%%%%%%%%%%%%%%%%%%%%%%%%

\emph{Case i).} $d(v'_x,v'_z)=d(x,z)$ (case $a$ or $b$ in Figure \ref{fig:xz}). Consider $\max\{\alpha, d(x,v'_x)\}$.

If $\max\{\alpha, d(x,v'_x)\}=\alpha$ (\emph{i.e.} $d(x,v'_x) \leq \alpha$), then Lemma \ref{Prop:xx'}, applied with $x'=v'_x$ (since $d(x,x')=d(x,v'_x)\le \alpha$), gives directly
$$
\delta^*(v'_x,y,z,w) = \delta^*(x,y,z,w).
$$

If $\max\{\alpha, d(x,v'_x)\}=d(x,v'_x)$ (\emph{i.e.} $d(x,v'_x) > \alpha$), take $x' \in [xy]$ such that $d(x,x')=\alpha$. Since $x'$ lies strictly between $x$ and $v'_x$ along $[xy]$, it remains in the same configuration as $x$ (case $a$ or $b$), so
$$
d(x',y)=d(x,y)-d(x,x')=d(x,y)-\alpha,
$$
$$
d(x',w)=d(x,w)+d(x,x')=d(x,w)+\alpha,
$$
$$
d(x',z)=d(x,z)-d(x,x')=d(x,z)-\alpha.
$$
Hence
%\begin{align*}
%\delta^*(x',y,z,w) &= \frac{1}{2}\left(d(x',z)+d(y,w)-d(x',y)-d(z,w)\right) \\
%&= \frac{1}{2}\left((d(x,z)-\alpha)+d(y,w)-(d(x,y)-\alpha)-d(z,w)\right) \\
%&= \frac{1}{2}\left(d(x,z)+d(y,w)-d(x,y)-d(z,w)\right) \\
%&= (x,y)_w-(x,z)_w = \delta^*(x,y,z,w),
%\end{align*}
%and, denoting $\alpha' := (y,z)_{w}-(x',y)_{w}$,
%\begin{align*}
%\alpha' &= \frac{1}{2}\left(d(x',y)+d(z,w)-d(x',w)-d(y,z)\right) \\
%&= \frac{1}{2}\left((d(x,y)-\alpha)+d(z,w)-(d(x,w)+\alpha)-d(y,z)\right) \\
%&= \frac{1}{2}\left(d(x,y)+d(z,w)-d(x,w)-d(y,z)\right) - \alpha \\
%&= \alpha - \alpha = 0.
%\end{align*}
%This reduces to the case $\alpha=0$ already treated above.
\begin{align*}
(x',y)-(x',z) &= \frac{1}{2} \left(  d(y,w) -  d(x',y) - d(z,w) + d(x',z)  \right) \\
&=\frac{1}{2} \left( d(y,w)- d(x,y)+ \alpha  -d(z,w)+ d(x,z)- \alpha  \right) \\
&=(x,y)_w-(x,z)_w= \delta^*(G). \\
(y,z)_w-(x',z)_w&= \frac{1}{2} \left( d(y,w)- d(y,z)-d(x',w)+ d(x',z)  \right) \\
&= \frac{1}{2} \left(d(y,w) - d(y,z)- d(x,w)- \alpha + d(x,z)- \alpha  \right) \\
&= (y,z)_w-(x,z)_w- \alpha \\
&=(x,y)_w-(x,z)_w = \delta^*(G).
\end{align*}
This reduces to the case $\alpha=0$ already treated above.

 \emph{Case ii).}  $d(v'_x,v''_z)=d(x,z)$ (case $c$ or $d$ in Figure \ref{fig:xz}). We take $x^* \in [xv'_x], \, z^* \in [zv'_z]$ such that  $d(x,x^*)=d(z,z^*)=\alpha/2$. Since $d(x^*,z^*)=d(x,z)-d(x,x^*)-d(z,z^*)$, we have $(x^*,y)_w-(x^*,z^*)_w=(y,z^*)_w-(x^*,z^*)_w=(x,y)_w-(x,z)_w =\delta^*(x,y,z,w)$, therefore we have  $\delta^*(x^*,y,z^*,w)=\delta^*(x,y,z,w)$ with  $\alpha'= \vert (x^*,y)_w - (y,z^*)_w \vert=0$ and this case was previously considered.
 
\smallskip

For case $y,w \not \in J'(G)$, we proceed symmetrically to the previous case.
 \end{proof}

%teorema 3 puntos
\begin{proposition}\label{Th:3puntos}
Let $G$ be a graph and $x,y,z,w \in G$ such that $\dx = (x,y)_w - (x,z)_w$. 
%$\d^*(G)=min \left\{ (x,y)_w , (y,z)_w ) \right\} -(x, z )_w \geq 0 $.\\
If $ \left \vert \{x,y,z,w \} \cap J'(G) \right \vert =1$ then there exist $x',y',z',w' \in J'(G)$ such that $\delta^*(x',y',z',w') \geq \delta^*(x,y,z,w)$.
%If  $z \in J^{\prime}$ then there exist $x^{\prime},y^{\prime}, w^{\prime} \in J^{\prime}$ such that 
%$$
%\delta^*(G)=\delta ^*(x',y',z,w').
%$$
\end{proposition}

\begin{proof}
Suppose that $z \in J'(G)$. 
Note that to prove this result it suffices to show that there exists $x\pr \in J\pr$ such that $\delta^*(x',y,z,w) \geq \delta^*(x,y,z,w)$ and Propositions \ref{Th:2puntos1} and \ref{Th:2puntos2} give the result. 

Since $x\notin J\pr $ we consider that $x \in v_0v_1 \in E(G)$. Let $v'_x, \, v''_x \in J'(v_0v_1)$ be such that $d(v_0,v'_x) < d(v_0,x)< d(v_0,v''_x)$ and  $d(v'_x, \, v''_x)=1/4$. 
 We assume $v_0, \, v'_x \in [xy]$.
If $v'_x \in [xz]$ we take $x'=v''_x$; otherwise, we take $x'=v'_x$. By Proposition \ref{Th:Pcuarto-Pgen} we have $d(x',z)=d(x,z)+d(x',x)$, $d(x',y) \leq d(x,y)+d(x,x')$, $d(x',w) \leq d(x,w)+d(x,x')$. So, we have 
\begin{align*}
(x',y)_w-(x',z)_w &= \frac{1}{2} \left( d(x',w)+ d(y,w)-d(x',y)-d(x',w)-d(z,w)+d(x',z) \right) \\
&\geq  \frac{1}{2} \left( d(y,w)-\left( d(x,y)+d(x,x') \right) -d(z,w)+ \left( d(x,z)+d(x,x') \right) \right) \\
 &=(x,y)_w-(x,z)_w .\\
 (y,z)_w-(x',z)_w &= \frac{1}{2} \left( d(y,w)+d(z,w)-d(y,z)-d(x',w)-d(z,w)+d(x',z) \right) \\
 & \geq \frac{1}{2} \left(  d(y,w)-d(y,z)-\left( d(x,w)+d(x,x') \right)+ \left( d(x,z) +d(x,x') \right) \right) \\
 &=  (y,z)_w-(x,z)_w \\
 &\geq (x,y)_w-(x,z)_w .
\end{align*}
So, we have
$$ \delta^*(x',y,z,w) \geq \min \{ (x',y)_w,(y,z)_w\}-(x',z)_w \geq (x,y)_w-(x,z)_w = \delta^*(x,y,z,w).$$
Since $\delta^*(x',y,z,w) \geq \delta^*(x,y,z,w)$, by Proposition \ref{Th:2puntos2} there exist $y', w' \in J'(G)$ such that
$$\delta^*(x',y',z,w') \geq \delta^*(x',y,z,w) \geq \delta^*(x,y,z,w). $$ 
\smallskip
The proof for the cases $x \in J'(G)$, $y \in J'(G)$ or $w \in J'(G)$ follows a similar way.
\end{proof}

\begin{proposition}\label{Th:4puntos}
Let $G$ be a graph and $x,y,z,w \in G$ such that $\dx=(x,y)_w-(x,z)_w$.
If $x,y,z,w \notin J^{\prime}$, then there exist $x^{\prime},y^{\prime},z^{\prime}, w^{\prime} \in J^{\prime}$ such that 
$$
\delta ^*(x',y',z',w') \geq \d^*(x,y,z,w).
$$
\end{proposition}

\begin{proof}
Suppose that $\d^*(x,y,z,w)=(x,y)_w-(x,z)_w$, $x \in v_0^xv_1^x$, $y \in v_0^yv_1^y$, $z \in v_0^zv_1^z$ and $w \in v_0^wv_1^w$, where  $v_0^xv_1^x, \, v_0^yv_1^y, \, v_0^zv_1^z, \,v_0^wv_1^w \in E(G)$.  Denote by  $v'_xv''_x, \,v'_yv''_y, \, v'_zv''_z, \, v'_wv''_w,$ the quarter segment of the edges where  $x$, $y$, $z$ and $w$ are contained respectively, \emph{i.e.}, $d(v^x_0,v'_x)<d(v^x_0,x)<d(v^x_0,v''_x)$; $d(v^y_0,v'_y)<d(v^y_0,y)<d(v^y_0,v''_y)$; $d(v^z_0,v'_z)<d(v^z_0,z)<d(v^z_0,v''_z)$ and  $d(v^w_0,v'_w)<d(v^w_0,w)<d(v^w_0,v''_w)$ with $v'_x,v''_x,v'_y,v''_y,v'_z,v''_z,v'_w,v''_w \in J'(G)$ and  $d(v'_x,v''_x)=d(v'_y,v''_y)=d(v'_z,v''_z)=d(v'_w,v''_w)=\frac{1}{4}$.\\
Without loss of generality, we can assume that $v'_x, v'_y \in [ xy ]$ and $v'_z, v'_w \in [ zw ]$.  By Proposition \ref{Prop:vinxyxw} we have that $v''_x,v''_w \in [xw]$ and $v''_y,v''_z \in [yz]$ (see Figure \ref{fig:4pts}).

%NO BORARR LO SIGUIENTE, PUEDE SER UTIL Y ES UNA EXPLICACION D ELO ANTERIOR
%If $v'_x \in [ xy ] \cap [ x w ] $ by Proposition \ref{Prop:vinxyxw} we have that $\d ^* (v''_x,y,z,w) \geq \d ^* (G)$, by Proposition \ref{Th:3puntos} there exist $y',z',w' \in J'(G)$ such that $\d^* (v''_x,y',z',w')=\d^*(G)$ obtaining the desired result, so we consider that $v''_x \in [xw]$. 
%If $v'_w \in [zw] \cap [xw]$ by Remark \ref{Rem:vinxyxw} we have that $\d^* (x,y,z,v''_w)\geq \d ^*(G)$ and by Proposition \ref{Th:3puntos} there exist $x',y',z' \in J'(G)$ such that $\d ^* (x',y',z',v''_w)=\d^*(G)$, so we consider that $v''_w \in [xw]$. 
%If $v'_z \in [zw]\cap [yz] $ by Remark \ref{Rem:vinxyxw} we have that $\d^* (x,y,v''_z,w)\geq \d ^*(G)$ and by Proposition \ref{Th:3puntos} there exist $x',y',w' \in J'(G)$ such that $\d ^* (x',y',v''_z,w')=\d^*(G)$, so we consider that $v''_z \in [yz]$. 
%If  $v'_y \in [xy]\cap [yz] $ by Remark \ref{Rem:vinxyxw} we have that $\d^* (x,v''_y,z,w)\geq \d ^*(G)$ and by Proposition \ref{Th:3puntos} there exist $x',z',w' \in J'(G)$ such that $\d ^* (x',v''_y,z',w')=\d^*(G)$, so we consider that $v''_y \in [yz]$.\\
We define
\begin{align*}\label{epsilons}
\e_1:=&min \{ d(x,v''_x),d(y,v'_y),d(z,v''_z),d(w,v'_w) \};\\
\e_2:=&min \{ d(x,v'_x),d(y,v''_y),d(z,v'_z),d(w,v''_w) \}.
\end{align*}
Taking $x' \in [ xv''_x ]$, $x'' \in [ xv'_x ]$, $y' \in [ yv'_y ] $, $y'' \in [ yv''_y ] $, $z' \in [ zv''_z ] $, $z'' \in [ zv'_z ] $,  $ w' \in [ wv'_w] $ and $ w'' \in [ w v''_w] $ such that
\begin{align*}
&d(x,x')=d(y,y')=d(z,z')=d(w,w')=\e _1; \\
&d(x,x'')=d(y,y'')=d(z,z'')=d(w,w'')=\e _2.
\end{align*}
By construction we have 
\begin{align*}
&d(x'',y'')=d(x',y')=d(x,y);\\
&d(x'',w'')=d(x',w')=d(x,w);\\
&d(y'',z'')=d(y',z')=d(y,z);\\
&d(z'',w'')=d(z',w')=d(z,w).\\
\end{align*}
Note that $\max \{ d(x'',z'')+d(y'',w''),d(x',z')+d(y',w') \} \geq d(x,z)+d(y,w)$. 
 If $\max \{ d(x'',z'')+d(y'',w''),d(x',z')+d(y',w') \}=d(x',z')+d(y',w')$, then  
\begin{align*}
(x',y')_{w'}-(x',z')_{w'}&= \frac{1}{2} \left( d(x',w')+d(y',w')-d(x',y')-d(x',w')-d(z',w')+d(x',z') \right)\\
&= \frac{1}{2} \left( d(x',z')+d(y',w')-d(x',y')-d(z',w') \right)\\
&\geq \frac{1}{2} \left( d(x,z)+d(y,w)-d(x',y')-d(z',w') \right) \\
&= \frac{1}{2} \left( d(x,z)+d(y,w)-d(x,y)-d(z,w) \right)\\
&=(x,y)_{w}-(x,z)_{w}. \\
(y',z')_{w'}-(x',z')_{w'}&= \frac{1}{2} \left( d(y',w')+d(z',w')-d(y',z')-d(x',w')-d(z',w')+d(x',z') \right)\\
&= \frac{1}{2} \left( d(x',z')+d(y',w')-d(x',w')-d(y',z') \right)\\
&\geq \frac{1}{2} \left( d(x,z)+d(y,w)-d(x',w')-d(y',z') \right) \\
&= \frac{1}{2} \left( d(x,z)+d(y,w)-d(x,w)-d(y,z) \right)\\
&=(y,z)_{w}-(x,z)_{w} \\
&\geq (x,y)_{w}-(x,z)_{w}.
\end{align*}
Therefore, $\delta^*(x',y',z',w') \geq \delta^*(x,y,z,w)$. \\
 If $\max \{ d(x'',z'') + d(y'',w''), d(x',z') + d(y',w') \} = d(x'',z'') + d(y'',w'')$, 
then a similar argument yields $\delta^*(x'',y'',z'',w'') \geq \delta^*(x,y,z,w)$.
Moreover, in each case, every element of the set $\{x', x'', y', y'', z', z'', w', w''\}$ belongs to $J'(G)$. Consequently, Propositions~\ref{Th:1punto}, \ref{Th:2puntos1}, \ref{Th:2puntos2}, and \ref{Th:3puntos} guarantee the existence of the remaining elements in $J'(G)$, which completes the proof.
\end{proof}

%Figura 2
\begin{figure}[htb]
\centering
\includegraphics[width=0.5\textwidth]{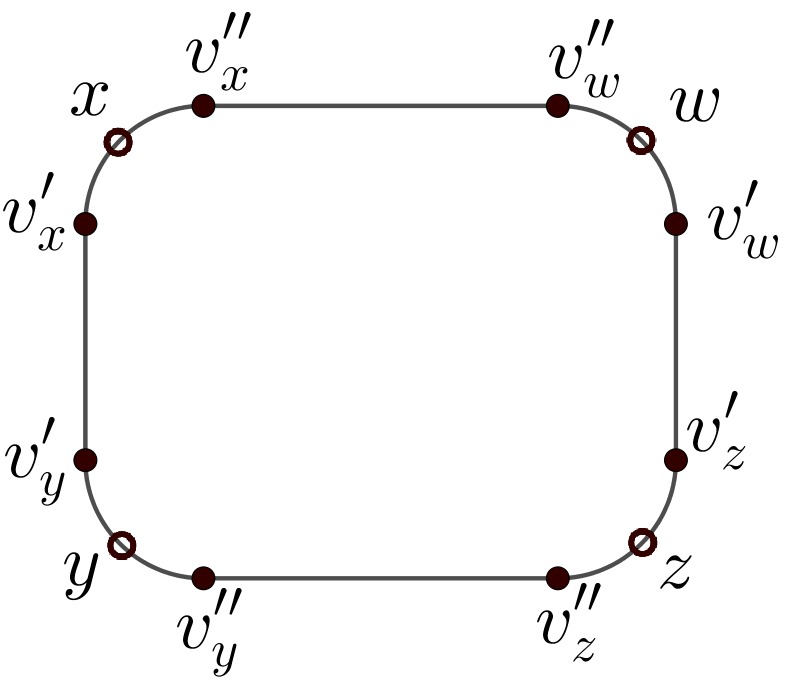}
\caption{Approximation of $x,y,z,w$ to the quarter points.}
\label{fig:4pts}
\end{figure}

%%%%%%%%%%%%%%%%%%%%%%%%%%%%%%%%%%%%%%%%%%%%%%%%%%%%%%%%%%%%%%%%%%%%%%%%%%%%%%%%%%%%%%%%%%%%%%%%%%%%%
%%%%%%%%%%%%%%%%%%%%%%%%%%%%%%%%%%%%%%%%%%%%%%%%%%%%%%%%%%%%%%%%%%%%%%%%%%%%%%%%%%%5555
%%%%%%%%%%%%%%%%%%%%%%%%%%%%%%%%%%%%%%%%%%%%%%%%%%%%%%%%%%%%%%%%%%%%%%%%%%%%%%%%%%%%%%%%%%%%%%%%%%%%%%5

%It should be noted that a natural to extend the result to infinite graphs; in this regard, we have proven the following.

%\begin{lemma}\label{lem:quantization}
%Let $G$ be any graph and let $x,y,z,w \in J'(G)$. Then $\delta^*(x,y,z,w) \in \tfrac18 \mathbb{Z}$.
%\end{lemma}
%\begin{proof}
%By construction, any $u \in J'(G)$ lies at one of the positions $0, \tfrac14, \tfrac12, \tfrac34$ along some edge of $G$. Consequently, for $u,v \in J'(G)$, the distance $d(u,v)$ -- realized as a minimum over finitely many candidate routes through the (at most four) endpoints of the edges containing $u$ and $v$, each contributing an integer (vertex-to-vertex) distance plus quarter-integer offsets at $u$ and $v$ -- is itself a multiple of $\tfrac14$. Hence $d(u,v) \in \tfrac14\mathbb{Z}$ for all $u,v \in J'(G)$.
%
%For $x,y,z,w \in J'(G)$, the Gromov product $(x,y)_w = \tfrac12\bigl(d(x,w)+d(y,w)-d(x,y)\bigr)$ is therefore a half of an element of $\tfrac14\mathbb{Z}$, i.e., $(x,y)_w \in \tfrac18\mathbb{Z}$; the same holds for every Gromov product appearing in the definition of $\delta^*(x,y,z,w)$. Since $\tfrac18\mathbb{Z}$ is closed under subtraction and under taking the minimum or maximum of finitely many of its elements, $\delta^*(x,y,z,w) \in \tfrac18\mathbb{Z}$.
%\end{proof}

\begin{proposition}\label{prop:Jprime}
Let $G$ be any graph and %set
$$
\delta^*_{J'}(G) := \sup\{\delta^*(x,y,z,w) : x,y,z,w \in J'(G)\}.
$$
If $\delta^*_{J'}(G) < \infty$, then there exist $x,y,z,w \in J'(G)$ such that $\delta^*_{J'}(G) = \delta^*(x,y,z,w)$.
\end{proposition}
\begin{proof}
Note that $\delta^*(x,y,z,w) \in \tfrac18 \mathbb{Z}$, since ,  any $u \in J'(G)$ lies at one of the positions $0, \tfrac14, \tfrac12, \tfrac34$ along some edge of $G$. Consequently, for $u,v \in J'(G)$, the distance $d(u,v)$ is itself a multiple of $\tfrac14$. %Hence $d(u,v) \in \tfrac14\mathbb{Z}$ for all $u,v \in J'(G)$.
Consider the set
$$
\mathcal{H} := \{\, 8\,\delta^*(x,y,z,w) \,:\, x,y,z,w \in J'(G) \,\}.
$$
%By Lemma~\ref{lem:quantization}, 
 $\mathcal{H} \subseteq \mathbb{Z}$; it is nonempty, since $J'(G) \neq \emptyset$, and it is bounded above by $8\,\delta^*_{J'}(G) < \infty$. Every nonempty subset of $\mathbb{Z}$ bounded above has a maximum, and since $\mathcal{H}$ is precisely the image of the map $(x,y,z,w) \mapsto 8\,\delta^*(x,y,z,w)$ on $J'(G)^4$, that maximum is realized by some quadruple $x_0,y_0,z_0,w_0 \in J'(G)$: $\max \mathcal{H} = 8\,\delta^*(x_0,y_0,z_0,w_0)$. Since
$$
8\,\delta^*_{J'}(G) = 8 \sup_{x,y,z,w \in J'(G)} \delta^*(x,y,z,w) = \sup \mathcal{H} = \max \mathcal{H} = 8\,\delta^*(x_0,y_0,z_0,w_0),
$$
dividing by $8$ gives $\delta^*_{J'}(G) = \delta^*(x_0,y_0,z_0,w_0)$.
\end{proof}

%%%%%%%%%%%%%%%%%%%%%%%%%%%%%%%%%%%%%%%%%%%%%%%%%%%%%%%%%%%%%%%%%%%%%%%%%%%%%%%%%%%%%%%%%%%%%%%%%%%%%%%%%
%%%%%%%%%%%%%%%%%%%%%%%%%%%%%%%%%%%%%%%%%%%%%%%%%%%%%%%%%%%%%%%%%%%%%%%%%%%%%%%%%%%%%%%%%%%%%%%%%%%%%%%%%

Propositions \ref{Th:1punto}, \ref{Th:2puntos1}, \ref{Th:2puntos2}, \ref{Th:3puntos}, \ref{Th:4puntos}, \ref{prop:Jprime} give the following result.

\begin{theorem}\label{PuntosCuartos}
For any graph $G$ (finite or infinite),
$$ \delta^*_{J'}(G) := \sup \left\{ \delta^*(x',y',z',w') : x',y',z',w' \in J'(G) \right\} = \delta^*(G). $$
\end{theorem}

\begin{proof}
The inequality $\delta^*_{J'}(G) \le \delta^*(G)$ is immediate, since $J'(G) \subseteq G$. For the reverse inequality, let $\varepsilon>0$ be arbitrary. By the definition of supremum, there exist $x,y,z,w \in G$ such that
$$
\delta^*(x,y,z,w) > \delta^*(G) - \varepsilon.
$$
By Propositions \ref{Th:1punto}, \ref{Th:2puntos1}, \ref{Th:2puntos2}, \ref{Th:3puntos}, \ref{Th:4puntos}, there exist $x',y',z',w' \in J'(G)$ satisfying
$$
\delta^*(x',y',z',w') \ge \delta^*(x,y,z,w) > \delta^*(G) - \varepsilon.
$$
Consequently, $\delta^*_{J'}(G) > \delta^*(G) - \varepsilon$. Since this holds for every $\varepsilon>0$, letting $\varepsilon \to 0^+$ yields $\delta^*_{J'}(G) \ge \delta^*(G)$, which, together with the reverse inequality above, completes the proof.
\end{proof}

%%%%%%%%%%%%%%%%%%%%%%%%%%%%%%%%%%%%%%%%%%%%%%%%%%%%%%%%%%%%%%%%%%%%%%%%%%%%%%%%%%%%%%%%%%%%%%%%%%%%%%%%%%%%%%%%%%%%%%%%

Theorem~\ref{PuntosCuartos} and Proposition~\ref{prop:Jprime} give the following result.

\begin{proposition}\label{prop:attainment-final}
Let $G$ be a graph (finite or infinite) with $\delta^*(G)<\infty$. Then there exist $x,y,z,w \in J'(G)$ such that
$$\delta^*(G)= \delta^*(x,y,z,w).$$
\end{proposition}
\begin{proof}
By Theorem~\ref{PuntosCuartos}, $\delta^*_{J'}(G)=\delta^*(G)<\infty$. Proposition~\ref{prop:Jprime} (whose proof places no restriction on the finiteness of $G$) then gives $x,y,z,w \in J'(G)$ with $\delta^*(x,y,z,w)=\delta^*_{J'}(G)=\delta^*(G)$.
\end{proof}

\section{Theoretical Applications of the Quarter-Point Theorem} \label{sec:applications}

The Quarter-Point Theorem (Theorem~\ref{PuntosCuartos}) establishes that computing the exact Gromov hyperbolicity constant $\delta^*(G)$ of a graph $G$ (finite or infinite) does not require evaluating the entire continuum of points in the metric graph $G$, but can be strictly reduced to the discrete set of quarter-points $J'(G)$. Beyond its clear algorithmic implications for reducing the computational search space, this discretization property serves as a powerful theoretical tool to derive tight theoretical bounds relating $\delta^*(G)$ to simpler or more restricted hyperbolic invariants.

In this section, we exploit the structural properties of $G$ and $J'(G)$ to address several fundamental problems in the metric theory of graphs. Specifically, we apply Theorem~\ref{PuntosCuartos} to:
\begin{enumerate}[(i)]
    \item Establish explicit comparison inequalities between the continuous hyperbolicity constant $\delta^*(G)$ and its restrictions to vertices ($\delta_V^*(G)$) and midpoints ($\delta_J^*(G)$), quantifying the precise maximum potential deviation when restricting the search space to these simpler subsets.
    \item Provide optimal extremal graph family examples (such as complete graphs and cycle graphs) that prove the sharpness of these theoretical bounds.
\end{enumerate}

Given a graph $G$, we denoted by  $\delta^*_V(G)$ the constant of hyperbolicity of the Gromov product restricted to the vertices of $G$, \emph{i.e.},
$$
\delta^*_V(G):= \sup \left \{ \min \{ (v_1,v_2)_{v_4},(v_2,v_3)_{v_4} \} - (v_1,v_3)_{v_4} : v_1,v_2,v_3,v_4 \in V(G)    \right \}.
$$

%The following result appears in \cite{Operadores}.
In \cite{Operadores} it was shown that
\begin{theorem}\label{Prop:Operadores}
Let $G$ be a graph, then   
$
\delta^*_V(G) \leq \delta^*(G) \leq \delta^*_V(G) +3.
$
\end{theorem}

The following result represents an improvement, now valid for any graph, finite or infinite.

\begin{theorem}\label{The:Gv+1}
Let $G$ be a graph, then 
$\d_V^*(G) \leq \delta^*(G) \leq \d_V^*(G) +1 $
 and these bounds are optimal.  
\end{theorem}
\begin{proof}
The lower bound is immediate, we will  demonstrate the upper bound. Note that if $x,y,z,w,x' \in G$  it is satisfied   $\d^*(x,y,z,w)\leq \delta^*(x',y,z,w) + d(x,x')$. 
By Proposition~\ref{prop:attainment-final} there exist $x,y,z,w \in J'(G)$ such that $\delta^*(G)=\dx=(x,y)_w-(x,z)_w $. Let $D=\{x,y,z,w\}$ be and $D_v:=\{ u \in V(G): d(u,v) \leq 1/2, \, v\in D \}$, in particular, if $u \in D$ we denoted by $u_v$ the element of $D_v$ such that $d(u,u_v) \leq 1/2$. Note that $u_v$   may not be unique. Following the idea of the proof of Proposition \ref{Th:4puntos} consider the cycle $C:=[xy]\cup [yz] \cup [zw] \cup [wx]$.
% $x_v,y_v,z_v,w_v$ the vertices of $G$ {\color{red}closest} to $x, \, y, \, z, \, w$ respectively ({\color{red} they may not be unique}). We have the following cases:

\emph{Case 1.} None of the elements of $D$ is a midpoint of any edge. We have that $ \vert u-u_v \vert \leq 1/4$ for every $u \in D$, so
\begin{align*}
\delta^*(G)=\dx  & \leq \delta^*(x_v,y_v,z_v,w_v)+ d(x,x_v)+d(y,y_v)+d(z,z_v)+d(w,w_v) \\
& \leq \delta^*(x_v,y_v,z_v,w_v) +1.
\end{align*}
%$$
%\delta^*(G)=\dx \leq \delta^*(x_v,y_v,z_v,w_v)+ d(x,x_v)+d(y,y_v)+d(z,z_v)+d(w,w_v) \leq \delta^*(x_v,y_v,z_v,w_v) +1.
%$$

\emph{Case 2.} One of the elements of $D$ is a midpoint of any edge. Suppose that $x$ is a midpoint.
 If $y_v \in [xy]$ then we choose $x_v \in [xw]$ and we have $d(x_v,w)=d(x,w)-1/2$ and $d(y_v,z) \leq d(y,z)+1/4$. 
 If $y_v \in [yz]$ we choose $x_v \in [xy]$ and we have $d(y_v,z)=d(y,z)-1/4$ and $d(x_v,w) \leq d(x,w)+1/2$. Note that in both cases we have $d(x_v,y_v) \leq d(x,y)+1/4$, $d(x_v,z) \geq d(x,z)-1/2$, $d(y_v,w) \geq d(y,w)-1/4$  and  $d(y_v,z)+d(x_v,w) \leq d(y,z)+d(x,w)+1/4$. We have 
 
   %$ d(x,z)+d(y,w)-3/4 \leq d(x_v,z)+d(y_v,w) \leq d(x,z)+d(y,w)+3/4$  and $d(y_v,z)+d(x_v,w) \leq d(y,z)+d(x,w)+1/4$. By Lemma \ref{Lem:dist-Gromov} we have  $d(y,z)+d(x,w) \leq d(x,y)+d(z,w)$, so
%Suppose that $x$ or $y $ is a midpoint, without loss of generality suppose that $x$ is a midpoint. We choose $x_v\in D_v$ such that $d(x_v,y_v) \leq d(x,y)+ 1/4$. By triangle inequality we have  $d(y_v,w)\geq d(y,w)-d(y,y_v)$,  $d(x_v,z) \geq d(x,z)-d(x,x_v)$,  so
\begin{align*}
(x_v,y_v)_w-(x_v,z)_w &= \frac{1}{2} \left( d(x_v,w)+d(w,y_v)-d(x_v,y_v)-d(x_v,w)-d(z,w)+d(x_v,z)  \right) \\
&=\frac{1}{2} \left(d(w,y_v)-d(x_v,y_v)-d(z,w)+d(x_v,z)   \right) \\
& \geq \frac{1}{2} \left(  \left(d(y,w)-1/4 \right) - \left(d(x,y)+1/4 \right) -\left( d(z,w) \right) + \left( d(x,z)-1/2 \right)   \right) \\
&=\frac{1}{2} \left( d(y,w)-d(x,y)-d(z,w)+d(x,z)-1 \right) \\
&=\frac{1}{2} \left( d(x,w)+d(y,w)-d(x,y)-d(x,w)-d(z,w)+d(x,z)-1  \right) \\
&= (x,y)_w-(x,z)_w - \frac{1}{2} \\
&=\delta^*(G)- \frac{1}{2}.\\
(y_v,z)_w-(x_v,z)_w&= \frac{1}{2} \left( d(y_v,w)+d(w,z)-d(y_v,z)-d(x_v,w)-d(w,z)+d(x_v,z)  \right) \\
&=\frac{1}{2} \left( d(y_v,w)+d(x_v,z)- \left( \, d(y_v,z)+d(x_v,w)\, \right) \right) \\
&\geq \frac{1}{2} \left( \left( d(y,w) -1/4  +d(x,z)-1/2 \right) - \left( d(y,z)+d(x,w)+1/4 \right)  \right) \\
&= \frac{1}{2} \left(  d(y,w) -d(y,z)- d(x,w)  +d(x,z)-1 \right) \\
&=\frac{1}{2} \left(  d(y,w)+d(z,w) -d(y,z)- \left( d(x,w) +d(z,w)  - d(x,z) \right) -1 \right) \\
&= (y,z)_w-(x,z)_w- \frac{1}{2}   \geq  (x,y)_w-(x,z)_w- \frac{1}{2} \\
 &= \delta^*(G)- \frac{1}{2}.
\end{align*}
So 
$
\delta^*(G)- \frac{1}{2} \leq \delta^*(x_v,y_v,z,w) \leq  \delta^*(x_v,y_v,z_v,w_v) +d(z,z_v)+d(w,w_v) \leq \delta^*_V(G)+\frac{1}{2},
$ 
therefore
$$
\delta^*(G) \leq \delta^*_V(G)+1.
$$
If the midpoint is $y$, $z$ or $w$ we follow a symmetrical procedure. %use the  change suggested in Remark \ref{Rem:1punto}.
%Si $y$, $z$ o $w$ fueran el punto medio podemos usar el cambio sugerido en el Remark \ref{Rem:1punto} para obtener el mismo resultado.

\emph{Case 3.} Two elements of $D$ are midpoints of any edge. We have the following subcases:

\emph{Case 3.1} $x$ and $z$ are midpoints.  We choose $x_v \in [xw]$ and $z_v \in [zw]$. We have $d(x_v,z_v) \geq d(x,z)-1$, $d(z_v,w)=d(z,w)-1/2$. $d(x_v,y) \leq d(x,y)+1/2$, $d(x_v,w)=d(x,w)-1/2$, $d(y,z_v)\leq d(y,z)+1/2$.
\begin{align*}
(x_v,y)_w-(x_v,z_v)_w&=\frac{1}{2} \left( d(x_v,w)+ d(y,w) - d(x_v,y)- \left(  d(x_v,w) + d(z_v,w)-  d(x_v,z_v) \right) \right) \\
&=\frac{1}{2} \left( d(y,w) - d(x_v,y)  - d(z_v,w)+  d(x_v,z_v)  \right) \\
&\geq \frac{1}{2} \left( d(y,w) -\left( d(x,y)+1/2 \right) - \left(d(z,w)-1/2  \right) + d(x,z)-1 \right) \\
&=\frac{1}{2} \left(  d(y,w) -  d(x,y) - d(z,w)+  d(x,z)-1 \right) \\
&=\frac{1}{2} \left( d(x,w)+  d(y,w) -  d(x,y) - \left( d(x,w)+ d(z,w) -   d(x,z) \right) -1 \right) \\
&=(x,y)_w-(x,z)_w- \frac{1}{2}.\\
(y,z_v)_w-(x_v,z_v)_w&=\frac{1}{2} \left( d(y,w) + d(z_v,w)-  d(y,z_v)- \left(  d(x_v,w) +  d(z_v,w) -  d(x_v,z_v) \right)  \right) \\
&=\frac{1}{2} \left( d(y,w) -  d(y,z_v)-  d(x_v,w)  +  d(x_v,z_v)   \right) \\
&=\frac{1}{2} \left( d(y,w)-(d(y,z)+1/2 )- (d(x,w)-1/2)+d(x,z)-1  \right) \\
&=\frac{1}{2} \left( d(y,w)- d(y,z) - d(x,w) +d(x,z)-1 \right) \\
&=\frac{1}{2} \left( d(y,w) +d(z,w)- d(y,z) - (d(x,w)+ d(z,w) - d(x,z))-1  \right) \\
&=(y,z)_w-(x,z)_w- \frac{1}{2} \\
&\geq  (x,y)_w-(x,z)_w- \frac{1}{2}.\\
%&=\frac{1}{2} \left(  \right) \\
%&\frac{1}{2} \left(  \right) \\
%&\frac{1}{2} \left(  \right) \\
\end{align*}

Therefore, 
$$
\delta^*(x_v,y,z_v,w) \geq \delta^*(x,y,z,w)- \frac{1}{2},
$$  
and so, 
$$
\delta^*(G)=\delta^*(x,y,z,w) \leq \delta^*(x_v,y,z_v,w) + \frac{1}{2} \leq \delta^*(x_v,y_v,z_v,w_v) + \frac{1}{2}+ d(y,y_v)+d(w,w_v) \leq \delta^*_V(G)+1.
$$
If $y$ and $w$ are the midpoints we follow a similar argument. % considering Remark \ref{Rem:2puntos2}.

\emph{Case 3.2} $x$ and $y$ are midpoints. We choose $x_v \in [xy]$ and $y_v \in [yz]$. We have $d(x_v,y_v)=d(x,y)$, $d(x_v,z) \geq d(x,z)-1/2$, $d(y_v,w) \geq d(y,w)-1/2$, $d(y_v,z)=d(y,z)-1/2$, $d(x_v,w) \leq d(x,w)+1/2$.
\begin{align*}
(x_v,y_v)_w-(x_v,z)_w &= \frac{1}{2} \left(d(y_v,w)-d(x_v,y_v)-d(z,w)+d(x_v,z)  \right) \\
& \geq \frac{1}{2} \left( d(y,w)-1/2-d(x,y)-d(z,w)+d(x,z)-1/2  \right) \\
& =\frac{1}{2} \left( d(y,w)-d(x,y)-d(z,w)+d(x,z) -1 \right) \\
&=(x,y)_w-(x,z)_w- \frac{1}{2}. \\
(y_v,z)_w-(x_v,z)_w &=\frac{1}{2} \left( d(y_v,w)-d(y_v,z)-d(x_v,w)+d(x_v,z) \right) \\
&\geq \frac{1}{2} \left( d(x,z)-1/2 - \left( d(y,z)-1/2 \right) -\left( d(x,w )+1/2 \right) +d(y,w)-1/2 \right) \\
&=\frac{1}{2} \left(  d(x,z) - d(y,z) -d(x,w ) +d(y,w)-1  \right) \\
&=(y,z)_w-(x,z)_w- \frac{1}{2} \geq (x,y)_w-(x,z)_w- \frac{1}{2}. \\
%&\frac{1}{2} \left(  \right) \\
%&\frac{1}{2} \left(  \right) \\
\end{align*}
Therefore, 
$$
\delta^*(x_v,y_v,z,w) \geq \delta^*(x,y,z,w)- \frac{1}{2},
$$
and so, 
$$
\delta^*(G)=\delta^*(x,y,z,w) \leq \delta^*(x_v,y_v,z,w) + \frac{1}{2} \leq \delta^*(x_v,y_v,z_v,w_v) + \frac{1}{2}+ \frac{1}{4} + \frac{1}{4} \leq \delta^*_V(G)+1.
$$
If $z,w$ are the midpoints we follow a similar argument. % considering Remark \ref{Rem:2puntosxy}.

\emph{Case 3.3} $x$ and $w$ are midpoints. We choose $x_v \in [xy]$ and $w_v \in [xw]$. We have $d(x_v,z) \geq d(x,z)-1/2$, $d(y,w_v) \geq d(y,w)-1/2$, $d(x_v,y)=d(x,y)-1/2$, $d(z,w_v) \geq d(z,w)+1/2$, $d(x_v,w_v)=d(x,w)$.
\begin{align*}
(x_v,y)_{w_v}-(x_v,z)_{w_v}&=\frac{1}{2} \left( d(y,w_v)-d(x_v,y)-d(z,w_v)+d(x_v,z)  \right) \\
&\geq \frac{1}{2} \left( d(y,w)-1/2 - \left(d(x,y)-1/2 \right)  - \left( d(z,w)+1/2 \right)+   d(x,z)-1/2  \right) \\
&=\frac{1}{2} \left( d(y,w) -d(x,y) - d(z,w) + d(x,z) -1 \right) \\
&=(x,y)_w-(x,z)_w - \frac{1}{2}. \\
(y,z)_{w_v}-(x_v,z)_{w_v} &= \frac{1}{2} \left( d(y,w_v)-d(y,z)-d(x_v,w_v)+d(x_v,z)  \right) \\
&\geq \frac{1}{2} \left( d(y,w)-1/2 -d(x,w) -d(y,z) + d(x,z)-1/2 \right) \\
&=(y,z)_{w}-(x_v,z)_{w} - \frac{1}{2} \geq  (x,y)_w-(x,z)_w - \frac{1}{2}.
%&\frac{1}{2} \left(  \right) \\
%&\frac{1}{2} \left(  \right) \\
%&\frac{1}{2} \left(  \right) \\
\end{align*}
Therefore, 
$$
\delta^*(x_v,y,z,w_v) \geq \delta^*(x,y,z,w)- \frac{1}{2},
$$
and so, 
$$
\delta^*(G)=\delta^*(x,y,z,w) \leq \delta^*(x_v,y,z,w_v) + \frac{1}{2} \leq \delta^*(x_v,y_v,z_v,w_v) + \frac{1}{2}+ \frac{1}{4} + \frac{1}{4} \leq \delta^*_V(G)+1.
$$
If $y$ and $z$ are the midpoints we follow a similar procedure. %considering Remark \ref{Rem:2puntosxw}.

\emph{Case 4.} Three elements of $D$ are midpoints of any edge. Suppose  that $x,y,z$ are midpoints. By \emph{Cases 2} and  \emph{3} we have 
$$\delta^*(G)=\delta^*(x,y,z,w) \leq \delta^*(x_v,y_v,z,w)+ \frac{1}{2} \leq  \delta^*(x_v,y_v,z_v,w_v)+ \frac{1}{2} + \frac{1}{2} \leq \delta^*_V(G)+1. $$

\emph{Case 5.} Four elements of $D$ are midpoints of any edge. By \emph{Case 3} we have 
$$\delta^*(G)=\delta^*(x,y,z,w) \leq \delta^*(x_v,y_v,z,w)+ \frac{1}{2} \leq  \delta^*(x_v,y_v,z_v,w_v)+ \frac{1}{2} + \frac{1}{2} \leq \delta^*_V(G)+1. $$

The lower bound %in Theorem \ref{The:Gv+1} 
is attained  for the cycle graph $C_4$ and the upper bound is attained for the complete graph $K_4$.
\end{proof}

%{\color{red} %inicia color rojo

Given a graph $G$, we denoted by  $\delta^*_J(G)$ the constant of hyperbolicity of the Gromov product restricted to the vertices of $G$ and to the midpoints of the edges of $G$, \emph{i.e.},
%Denotamos por $\delta^*_V(G)$ la constante de hiperbolicidad del producto de Gromov restringida a los v\'ertices de $G$, es decir:
$$
\delta^*_J(G):= \sup \left \{ \min \{ (v_1,v_2)_{v_4},(v_2,v_3)_{v_4} \} - (v_1,v_3)_{v_4} : v_1,v_2,v_3,v_4 \in J(G)    \right \}.
$$

%The proof of the following result follows the same procedure as that of Proposition \ref{Prop:Gv+2}.
%
%\begin{proposition}\label{prop:Jbound}
%Let $G$ be a graph, then
%$$
%\delta^*_{J}(G) \leq \delta^*(G) \leq \delta^*_J(G)+1. 
%$$
%\end{proposition}
%\begin{proof}
%The lower bound is immediate, we will  demonstrate the upper bound. Note that if $x,y,z,w,x' \in G$  it is satisfied   $\d^*(x,y,z,w)\leq \delta^*(x',y,z,w) + d(x,x')$. Let $x_1,x_2,x_3,x_4 \in G$ be such that $\delta^*(x_1,x_2,x_3,x_4)=\delta^*(G)$. There exist $x^J_1,x^J_2,x^J_3,x^J_4 \in J(G)$ such that $d(x_i,x^J_i) \leq 1/4$, $1 \leq i \leq 4$, therefore
%$$
%\delta^*(G)= \delta^*(x_1,x_2,x_3,x_4) \leq \delta^*(x^J_1,x_2,x_3,x_4) +1/4 \leq \dots \leq \delta^*(x^J_1,x^J_2,x^J_3,x^J_4)+1 \leq \delta_J^*(G)+1.
%$$
%\end{proof}

%The following result is an improvement on the previous one.

\begin{theorem}\label{Th:mejora2}
Let $G$ be a graph, then
$$ 
\delta^*_{J}(G) \leq \delta^*(G) \leq \delta^*_J(G)+ \frac{1}{2}. 
$$
\end{theorem}
\begin{proof}
The lower bound is immediate. We will prove the upper bound. 
By Proposition~\ref{prop:attainment-final} there are $x,y,z,w \in J'(G)$ such that $\delta^*(G)= \dx$, suppose that $\dx =(x,y)_w-(x,z)_w$.
Consider $x',y',z',w' \in J(G)$ such that $d(x,x') \leq 1/4$, $d(y,y') \leq 1/4$, $d(z,z') \leq 1/4$, $d(w,w') \leq 1/4$, $d(x',z')=d(x,z)$, $d(y',w')=d(y,w)$. 
We have that $d(x',y') \leq d(x,x')+ d(x,z)+ d(z,z')=d(x,z)+ 1/2$. Similarly, we have $d(x',w') \leq d(x,w) + 1/2$, $d(y',z') \leq d(y,z)+1/2$ and $d(w',z') \leq d(w,z)+1/2$.
We have
\begin{align*}
(x',y')_{w'}-(x',z')_{w'} &= \frac{1}{2} \left(d(x',w')+ d(y',w')- d(x',y') - d(x',w')-d(z',w') +d(x',z')\right) \\
& \geq \frac{1}{2} \left( d(y,w)- \left( d(x,y) + \frac{1}{2} \right) - \left( d(z,w) + \frac{1}{2} \right) +d(x,z)  \right) \\
&=(x,y)_w-(x,z)_w - \frac{1}{2}\\
 &= \delta^*(G)- \frac{1}{2}. \\
(y',z')_{w'}-(x',z')_{w'} &= \frac{1}{2} \left( d(y',w')+d(z',w')- d(z',y') - d(x',w')-d(z',w') +d(x',z')\right) \\
& \geq \frac{1}{2} \left( d(y,w) - \left( d(y,z) + \frac{1}{2} \right) - \left( d(x,w) + \frac{1}{2} \right) +d(x,z)  \right) \\
&= (y,z)_w-(x,z)_w - \frac{1}{2} \\
&\geq   (x,y)_w-(x,z)_w - \frac{1}{2} \\
&= \delta^*(G)- \frac{1}{2}.
\end{align*}
So 
$$
\delta_J^*(G) \geq \min \{ (x',y')_{w'}, (y',z')_{w'} \} - (x',z')_{w'} \geq  (x,y)_w-(x,z)_w - \frac{1}{2} = \delta^*(G)- \frac{1}{2}.
$$
Therefore
$$
\delta^*(G) \leq \delta_J^*(G) + \frac{1}{2}. 
$$
\end{proof}

\begin{corollary}
Let $G$ be a graph, then
$$
\delta_V^*(G) \leq \delta_J^*(G) \leq \delta^*(G) \leq  \min \{\delta_J^*(G)+ \frac{1}{2}, \delta_V^*(G)+1 \}.  
$$
\end{corollary}

The previous Corollary might lead one to think that we could obtain the relation 
$$
\delta^*_J (G) \leq \delta^*_V(G)+1/2,
$$
however, this is false, since $\delta^*_V(K_4)=0$ and $\delta^*_J(K_4)=1$.

%PREGUNTA: lo siguiente ser\'ia cierto? Dijo Jesus que no, pero no recuerdo el ejemplo
%
%$$
%\delta_J^*(G)+ \frac{1}{2} \leq \delta_V^*(G)+1. 
%$$

%} %termina color rojo

\section{Conclusions}

In this paper we proved the Quarter-Point Theorem (Theorem~\ref{PuntosCuartos}): for any graph $G$, finite or infinite, the exact Gromov hyperbolicity constant $\delta^*(G)$, in principle defined as a supremum over quadruples of points ranging over the whole metric graph, is in fact attained on the discrete subset $J'(G)$ of vertices, edge midpoints, and quarter-points (Proposition~\ref{prop:attainment-final}). As a consequence, we obtained sharper structural bounds, valid for every graph, comparing $\delta^*(G)$ with its restrictions to vertices, $\delta_V^*(G)$, and to vertices and midpoints, $\delta_J^*(G)$ (Theorem~\ref{The:Gv+1}, Theorem~\ref{Th:mejora2}), improving the previously known bound $\delta_V^*(G) \le \delta^*(G) \le \delta_V^*(G)+3$ (Theorem~\ref{Prop:Operadores}) to $\delta_V^*(G) \le \delta^*(G) \le \delta_V^*(G)+1$, together with the analogous bound $\delta_J^*(G) \le \delta^*(G) \le \delta_J^*(G)+\tfrac12$.

Beyond the sharper bounds themselves, the discretization result reduces the computation of $\delta^*(G)$ from a continuous optimization problem to a finite (or, for infinite $G$, discrete) combinatorial search over $J'(G)$, which is of independent algorithmic interest for computing hyperbolicity in networks and metric graphs.

%Propositions \ref{Th:1punto}, \ref{Th:2puntos1}, \ref{Th:2puntos2}, \ref{Th:3puntos} and \ref{Th:4puntos} are now stated for an arbitrary quadruple $x,y,z,w \in G$, rather than one already attaining $\delta^*(G)$; consequently, Theorem~\ref{PuntosCuartos} holds unconditionally, for finite and infinite graphs alike.
%
%Some questions remain open. First, explicit extremal examples establishing the sharpness of the bounds in Theorem~\ref{Th:mejora2} ($\delta^*(G)\le \delta_J^*(G)+\tfrac12$) are still needed, beyond the cycle $C_4$ and complete graph $K_4$ already used to show optimality of the bounds in Theorem~\ref{The:Gv+1}. 
An open problem remains: we still need explicit extreme examples to demonstrate the accuracy of the bounds in Theorem~\ref{Th:mejora2} ($\delta^*(G)\le \delta_J^*(G)+\tfrac12$), in fact, we believe that this bound can be improved to $\delta^*(G)\le \delta_J^*(G)+\tfrac14$.

\section*{Acknowledgements}


\begin{thebibliography}{99}

\bibitem{BRS} Bermudo, S., J. M. Rodr\'iguez, and J. M. Sigarreta, Computing the hyperbolicity constant, {\it Comput. Math. Appl.} 62 (2011), 4592--4595.
https://doi.org/10.1016/j.camwa.2011.10.041


%capitulo de libro
\bibitem{BHB} B. H. Bowditch, Notes on Gromov's hyperbolicity criterion for path-metric spaces in E. Ghys, A. Haefliger, A. Verjovsky (Eds.), Group Theory from a Geometrical Viewpoint,  World Scientific, River Edge, NJ, 1991,  64--167.
%\bibitem{BHB} B. H. Bowditch, Notes on Gromov's hyperbolicity criterion for path-metric spaces in E. Ghys, A. Haefliger, A. Verjovsky (Eds.), Group Theory from a Geometrical Viewpoint, Trieste, 1990; World Scientific, River Edge, NJ, 1991 64--167.
%https://doi.org/10.1142/1235 No estoy seguro que sea este


\bibitem{ABCD}
Alonso, J.; Brady, T.; Cooper, D.; Delzant, T.; Ferlini, V.; Lustig, M.; Mihalik, M.; Shapiro, M.; Short, H.
Notes on word hyperbolic groups. In Group Theory from a Geometrical Viewpoint; Ghys, E., Haefliger, A.,
Verjovsky, A., Eds.;World Scientific: Singapore, 1992; pp. 3--63.


\bibitem{Ghys} Ghys, E.; de la Harpe, P. Sur les Groupes Hyperboliques d'apr\`es Mikhael Gromov; Progress in Mathematics 83;
Birkh\"auser Boston Inc.: Boston, MA, USA, 1990.


\bibitem{Operadores} J. A. M\'endez-Berm\'udez, Rosalio Reyes, Jos\'e M. Rodr\'iguez and Jos\'e M. Sigarreta, Hyperbolicity on Graph Operators, {\it symmetry}  10, 2018, 360



%book
\bibitem{G1} M. Gromov. Hyperbolic groups. In {\em Essays in Group Theory};
Gersten, S.M., Ed.; Mathematical Sciences Research Institute Publications; Springer: Berlin, Germany, {\bf 8} (1987),  75--263.


\bibitem{O} K. Oshika.  Discrete groups. \emph{AMS Bookstore}; {Providence, RI, USA, 2002.}


\bibitem{BRSV} S. Bermudo,J. M. Rodr\'{\i}guez, J. M. Sigarreta  and J.-M. Vilaire. Gromov hyperbolic graphs,   {\it Discr. Math.} {\bf  313} (2013), 1575--1585.
https://doi.org/10.1016/j.disc.2013.04.009

\bibitem{BeRoSiVi1}
S. Bermudo, J. M. Rodr\'iguez, J. M. Sigarreta and J.-M. Vilaire.
\newblock Mathematical properties of Gromov hyperbolic graphs,
\newblock {\it  AIP Conf. Proc.}
 {\bf 1281} (2010), 575--578.
 https://doi.org/10.1063/1.3498541


\bibitem{BeRoSiTo}
 S. Bermudo, J. M. Rodr\'iguez, J. M. Sigarreta  and E. Touris.
\newblock Hyperbolicity and complement of graphs,
\newblock {\it Appl. Math. Letters.}
{\bf 24} (2011), 1882--1887.

\bibitem{BHB1}
G. Brinkmann, J. Koolen and V. Moulton.
\newblock On the hyperbolicity of chordal graphs,
\newblock {\it Ann. Comb.} {\bf 5(1)} (2001), 61--69.
DOI: 10.1007/s00026-001-8007-7


\bibitem{CPRS}
W. Carballosa, D. Pestana, J. M. Rodr{\'\i}guez and J. M. Sigarreta.
\newblock Distortion of the hyperbolicity constant of a graph.
\newblock {\it Electr. J. Combin.}
 {\bf 19} (2012), P67.
 https://doi.org/10.37236/2175

\bibitem{CRSV} W. Carballosa, J. M.  Rodr\'{\i}guez, J. M. Sigarreta and M.  Villeta. On the hyperbolicity constant of Line Graphs, {\it Electr. J. Combin.} {\bf 18} (2011), P210.
https://doi.org/10.37236/697

\bibitem{K50} R. Frigerio and A. Sisto. Characterizing hyperbolic spaces and real trees,
{\it Geom. Dedicata.} {\bf 142} (2009), 139--149.
DOI: 10.1007/s10711-009-9363-4



\bibitem{GL} C. Gavoille and O. Ly.  Distance Labeling in Hyperbolic Graphs. In: Deng X., Du DZ. (eds) Algorithms and Computation. ISAAC 2005. Lecture Notes in Computer Science. Springer, Berlin, Heidelberg. {\bf  3827} (2005), 1071--1079. https://doi.org/10.1007/11602613\_106

\bibitem{K27}
E. Jonckheere and P. Lohsoonthorn.
\newblock A hyperbolic geometry approach to multipath routing,
\newblock Proceedings of the 10th Mediterranean Conference
on Control and Automation (MED 2002), Lisbon, Portugal, July 2002.
FA5-1.

\bibitem{K22} E. A. Jonckheere, P. Lohsoonthorn. Geometry of network security.
In Proceedings of the American Control Conference. (2004), 111--151.
% Boston, MA, USA, 30 June--2 July 2004;  111--151.

\bibitem{K24}
E. A. Jonckheere, P. Lohsoonthorn and F. Bonahon.
\newblock Scaled Gromov hyperbolic graphs,
\newblock {\it J. Graph Theory.}  {\bf 57(2)} (2008), 157--180.
https://doi.org/10.1002/jgt.20275



\bibitem{KoMo} J. H. Koolen and V. Moulton. Hyperbolic Bridged Graphs, {\it Europ. J. Comb.} {\bf 23} (2002), 683--699.
DOI: 10.1006/eujc.2002.0591

\bibitem{MiRoSiVi1}
J. Michel, J. M. Rodr{\'\i}guez, J. M. Sigarreta and M. Villeta.
\newblock Hyperbolicity and parameters of graphs,
\newblock {\em Ars Combin.}
{\bf 100} (2011), 43--63.
%NO doi

\bibitem{MiRoSiVi2}
J. Michel, J. M. Rodr\'{\i}guez, J. M. Sigarreta and M. Villeta.
\newblock Gromov hyperbolicity in Cartesian product graphs,
\newblock {\em Proc. Indian Acad. Sci. (Math. Sci.)}
{\bf 120(5)} (2010), 593--609.
DOI: 10.1007/s12044-010-0048-6

\bibitem{PRSV2}
D. Pestana, J. M. Rodr\'{\i}guez, J. M., Sigarreta and M. Villeta.
\newblock Gromov hyperbolic cubic graphs,
\newblock {\em Central Europ. J. Math.}
{\bf 10} (2012), 1141--1151.
https://doi.org/10.2478/s11533-012-0036-4


\bibitem{PRT1}
A. Portilla, J. M. Rodr\'{\i}guez,  and E. Tour\'is.
\newblock Gromov hyperbolicity through decomposition of metric spaces II,
\newblock {\em J. Geom. Anal.}
 {\bf 14(1)} (2004), 123--149.
 DOI: 10.1007/BF02921869

\bibitem{PRT2}
A. Portilla, J. M. Rodr\'{\i}guez and E. Tour\'is.
\newblock Stability of Gromov hyperbolicity,
\newblock {\em J. Adv. Math. Studies.}
 {\bf 2(2)} (2009),  77--96.
 %NO DOI

\bibitem{PT}
A. Portilla and E. Tour{\'\i}s.
\newblock A characterization of Gromov hyperbolicity of surfaces with variable
negative curvature,
\newblock {\em Publ. Mat.}
  {\bf 53(1)} (2009), 83--110.
  DOI: 10.5565/PUBLMAT\_53109\_04

\bibitem{RS}
J. M. Rodr\'{\i}guez and J. M. Sigarreta.
\newblock Bounds on Gromov hyperbolicity constant in graphs,
\newblock {\em Proc. Indian Acad. Sci. (Math. Sci.)}
{\bf 122} (2012), 53--65.
https://doi.org/10.1007/s12044-012-0060-0

\bibitem{RSVV} J. M. Rodr\'{\i}guez, J. M. Sigarreta, J.-M Vilaire, J.-M. and M. Villeta. On the hyperbolicity constant in graphs, {\it Discr. Math}. {\bf  311} (2011), 211--219.
https://doi.org/10.1016/j.disc.2010.11.005


\bibitem{T}
E. Tour{\'\i}s.
\newblock Graphs and Gromov hyperbolicity of non-constant negatively curved surfaces,
\newblock {\em J. Math. Anal. Appl.}
 {\bf 380(2)} (2011), 865--881.
https://doi.org/10.1016/j.jmaa.2011.02.067

\bibitem{WZ} Y. Wu and C. Zhang.
Chordality and hyperbolicity of a graph, {\it Electr. J. Comb.} {\bf 18(1)} (2011). %, P43.
https://doi.org/10.37236/530



%\bibitem{Soto} Soto, M. Quelques Propri¨¦t¨¦s Topologiques des Graphes et Applications a Internet et aux R¨¦seaux. Ph.D. Thesis, Universit¨¦ Paris Diderot, Paris, France, 2011.

%%book
%\bibitem{GH} E. Ghys, P.  de la Harpe, Sur les Groupes Hyperboliques d'apr\`es Mikhael Gromov,  Progress in Mathematics 83, Birkh\"auser Boston Inc., Boston, MA, 1990.

%\bibitem{BeRoSi}
%S. Bermudo, J. M. Rodr\'iguez and J. M. Sigarreta.
%\newblock Computing the hyperbolicity constant,
%\newblock {\em Comput Math. Appl.}
%{\bf 62} (2011), 4592--4595.
%https://doi.org/10.1016/j.camwa.2011.10.041




\end{thebibliography}
\end{document}